\documentclass[a4paper]{amsart}
\pdfoutput=1

\usepackage[utf8]{inputenc}
\usepackage[T1]{fontenc}
\usepackage{lmodern}
\usepackage{amssymb,amsxtra}
\usepackage{nicefrac,mathtools,enumitem}
\setlist[enumerate,1]{label=\textup{(\arabic*)}}% ensure that enumerations in theorems are upright

\usepackage{tikz}
\usetikzlibrary{matrix}
\tikzset{cd/.style=matrix of math nodes,row sep=2em,column sep=2em, text height=1.5ex, text depth=0.5ex}
\tikzset{cdar/.style=->,auto}
\tikzset{overar/.style={draw=white,double=black,double distance=.4pt,very thick}}

\usepackage{microtype}
\usepackage[pdftitle={Braided free orthogonal quantum groups},
  pdfauthor={Ralf Meyer, Sutanu Roy},
  pdfsubject={Mathematics; MSC }
]{hyperref}
\usepackage[lite]{amsrefs}

\renewcommand{\PrintDOI}[1]{\href{http://dx.doi.org/\detokenize{#1}}{doi: \detokenize{#1}}}

\BibSpec{book}{%
  +{}  {\PrintPrimary}                {transition}
  +{,} { \textit}                     {title}
  +{.} { }                            {part}
  +{:} { \textit}                     {subtitle}
  +{,} { \PrintEdition}               {edition}
  +{}  { \PrintEditorsB}              {editor}
  +{,} { \PrintTranslatorsC}          {translator}
  +{,} { \PrintContributions}         {contribution}
  +{,} { }                            {series}
  +{,} { \voltext}                    {volume}
  +{,} { }                            {publisher}
  +{,} { }                            {organization}
  +{,} { }                            {address}
  +{,} { \PrintDateB}                 {date}
  +{,} { }                            {status}
  +{}  { \parenthesize}               {language}
  +{}  { \PrintTranslation}           {translation}
  +{;} { \PrintReprint}               {reprint}
  +{.} { }                            {note}
  +{.} {}                             {transition}
  +{} { \PrintDOI}                   {doi}
  +{} { available at \url}            {eprint}
  +{}  {\SentenceSpace \PrintReviews} {review}
}

\numberwithin{equation}{section}

\theoremstyle{plain}
\newtheorem{theorem}[equation]{Theorem}
\newtheorem{lemma}[equation]{Lemma}
\newtheorem{proposition}[equation]{Proposition}

\newtheorem{corollary}[equation]{Corollary}

\theoremstyle{definition}

\theoremstyle{remark}
\newtheorem{remark}[equation]{Remark}

\newcommand{\tenscorep}{\mathbin{\begin{tikzpicture}[baseline,x=.75ex,y=.75ex] \draw (-0.8,1.15)--(0.8,1.15);\draw(0,-0.25)--(0,1.15); \draw (0,0.75) circle [radius = 1];\end{tikzpicture}}}

\newcommand*{\Braiding}[2]{\begin{tikzpicture}[baseline]
    \draw[-] (0,0) -- (1.4ex,1.4ex) node[right,inner sep=0pt] {$\scriptstyle #2$};
    \draw[-,draw=white,line width=2.4pt] (0,1.4ex) -- (1.4ex,0);
    \draw[-] (1.4ex,0) -- (0,1.4ex) node[left,inner sep=0pt] {$\scriptstyle #1$};
  \end{tikzpicture}}
\newcommand*{\Dualbraiding}[2]{\begin{tikzpicture}[baseline]
    \draw[-] (1.4ex,0) -- (0,1.4ex) node[left,inner sep=0pt] {$\scriptstyle #1$};
    \draw[-,draw=white,line width=2.4pt] (0,0) -- (1.4ex,1.4ex);
    \draw[-] (0,0) -- (1.4ex,1.4ex) node[right,inner sep=0pt] {$\scriptstyle #2$};
  \end{tikzpicture}}

\newcommand*{\Corep}[1]{\mathbb{#1}}          %Corepresentation as operator on Hilbert space
\newcommand*{\nb}{\nobreakdash}
\newcommand*{\Star}{$^*$\nb-}

\newcommand*{\C}{\mathbb C}
\newcommand*{\Z}{\mathbb Z}

\newcommand*{\N}{\mathbb N}

\newcommand*{\T}{\mathbb T}
\newcommand*{\Mat}{\mathbb M}

\newcommand*{\G}[1][G]{\mathbb #1}% quantum group

\newcommand*{\qSU}{\textup{SU}_q(2)}% quantum SU(2)
\newcommand*{\qU}{\textup{U}_q(2)}% quantum SU(2)

\newcommand*{\Comult}[1][]{\Delta_{#1}}%comultiplication
\newcommand*{\Rep}{\mathrm{Rep}}%faithful representations
\newcommand*{\Bound}{\mathbb B}%adjointable operators on a Hilbert module
\newcommand*{\Comp}{\mathbb K}%compact operators on a Hilbert module
\newcommand*{\congto}{\xrightarrow\sim}%isomorphism to ...

\newcommand*{\ima}{\textup i}%imaginary unit
\newcommand*{\Cont}{\textup C}%continuous functions

\newcommand*{\Mor}{\textup{Mor}}%nondegenerate *-homomorphisms of C*-algebras
\newcommand*{\Id}{\textup{id}}%identity map

\newcommand*{\Rmat}{\textup R}%R-matrix
\newcommand*{\Cst}{\textup C^*}%C*-algebra
\newcommand*{\Hils}[1][H]{\mathcal{#1}}%Hilbert space
\newcommand*{\U}{\mathcal U}%unitary group

\newcommand*{\defeq}{\mathrel{\vcentcolon=}}

\DeclarePairedDelimiter{\abs}{\lvert}{\rvert}% absolute value
\DeclarePairedDelimiter{\norm}{\lVert}{\rVert}% norm
\DeclarePairedDelimiterX{\setgiven}[2]{\{}{\}}{#1\,{:}\,\mathopen{}#2}% set given by

\newcommand*{\conj}[1]{\overline{#1}}

\DeclareMathOperator{\Endo}{End}% endomorphism ring of a vector space

\usepackage{xcolor}

\begin{document}
\title{Braided free orthogonal quantum groups}

\author{Ralf Meyer}
\email{rmeyer2@uni-goettingen.de}
\address{Mathematisches Institut\\
  Georg-August Universität Göttingen\\
  Bunsenstraße 3--5\\
  37073 Göttingen\\
  Germany}

\author{Sutanu Roy}
\email{sutanu@niser.ac.in}
\address{School of Mathematical Sciences\\
 National Institute of Science Education and Research  Bhubaneswar, HBNI\\
 Jatni, 752050\\
 India}

\begin{abstract}
  We construct some braided quantum groups over the circle group.  These
  are analogous to the free orthogonal quantum groups and generalise
  the braided quantum SU(2) groups for complex deformation
  parameter.  We describe their irreducible representations and
  fusion rules and study when they are monoidally equivalent.
\end{abstract}

\subjclass[2000]{81R50}
\keywords{quantum group, braided quantum group, free orthogonal
  quantum group, quantum U(2) group, monoidal equivalence}

 \thanks{The second author was partially supported by INSPIRE faculty award given by D.S.T., Government of India grant no.
 DST/INSPIRE/04/2016/000215}

\maketitle

\section{Introduction}
\label{sec:introduction}

The quantum \(\textup{SU}(2)\) groups for a real deformation
parameter~\(q\) are
archetypes for the theory of compact quantum groups.  These are
generalised in~\cite{Kasprzak-Meyer-Roy-Woronowicz:Braided_SU2} to
the case \(q\in\C^\times\).  The underlying
\(\Cst\)\nb-algebra~\(B\) of \(\qSU\) for
\(q\in\C^\times\) is the universal \(\Cst\)\nb-algebra with two
generators \(\alpha,\gamma\), subject to the relations
\(\alpha^*\alpha + \gamma^*\gamma = 1\), \(\alpha\alpha^* +
\abs{q}^2\gamma^*\gamma = 1\), \(\gamma\gamma^* = \gamma^*\gamma\),
\(\alpha\gamma = \conj{q} \gamma\alpha\) and \(\alpha\gamma^* =
q\gamma^*\alpha\).  The comultiplication no longer takes values
in~\(B\otimes B\), but in a certain twisted tensor product
\(B\boxtimes B\).  This twisted tensor product uses
a parameter \(\zeta\in \T\) and the action of
the circle group~\(\T\) on~\(B\) defined by
\(\varrho_z(\alpha)=\alpha\) and \(\varrho_z(\gamma)=z\gamma\) for all
\(z\in\T\).

For real~\(q\),
\(\qSU\)
is the universal quantum group with a representation~\(S\)
on~\(\C^2\)
where a certain vector in \(\C^2\otimes \C^2\) is
invariant under the tensor product representation \(S\tenscorep S\).
Replacing~\(\C^2\)
by~\(\C^n\)
gives the free orthogonal quantum groups by Van Daele and
Wang~\cite{Daele-Wang:Universal}.
The braided quantum group \(\qSU\)
for complex~\(q\)
has a similar universal property (see
\cite{Kasprzak-Meyer-Roy-Woronowicz:Braided_SU2}*{Theorem 5.4}).
The difference is that~\(\C^2\)
now carries a non-trivial representation of the circle group, which
modifies the formula for \(S\tenscorep S\)
and thus the condition for a vector to be invariant.  In this
article, we are going to construct braided compact quantum
groups over the circle group~\(\T\)
that generalise both the free orthogonal quantum groups and the
braided \(\qSU\)
of~\cite{Kasprzak-Meyer-Roy-Woronowicz:Braided_SU2}.

The data to define them is a triple \((V,\pi,\omega)\)
consisting of a finite-dimensional vector space~\(V\)
with a representation~\(\pi\)
of~\(\T\)
and a \(\T\)\nb-homogeneous
vector \(\omega\in V\otimes V\),
subject to a condition that requires some notation.  Let
\(V_k\subseteq V\)
for \(k\in\Z\)
be the subspace of \(k\)\nb-homogeneous
elements of~\(V\).
Choose bases for each~\(V_k\)
so that a vector \(\omega\in V\otimes V\)
corresponds to a matrix.  Write this matrix as a block matrix
\(\Omega = (\Omega_{i j})\)
according to the decomposition \(V=\bigoplus V_k\).
Assume that~\(\omega\)
is \(d\)\nb-homogeneous
for \(d\in\Z\).
Then \(\Omega_{i j}=0\)
unless \(i+j=d\).
To construct a braided quantum group, we need~\(\Omega\)
to satisfy the following condition: there is a non-zero scalar
\(c\in\C^\times\) with
\begin{equation}
  \label{eq:Xi_condition}
  \conj{\Omega_{i,d-i}} \cdot \Omega_{d-i,i}
  = c\cdot \zeta^{d\cdot i}\cdot 1
  \qquad\text{for all }i\in\Z,
\end{equation}
where~\(1\) denotes the identity matrix of appropriate size
and~\(\zeta\) is the twisting parameter for the braided tensor
product.  Then the matrices~\(\Omega_{i,d-i}\) are invertible for all
\(i\in\Z\),
and \(\conj{c}/c = \zeta^{d^2}\)
is needed for the equations for \(i\) and \(d-i\) to be compatible.
Combining the blocks~\(\Omega_{ij}\) in one matrix~\(\Omega\), we may
rewrite~\eqref{eq:Xi_condition} more briefly as
\begin{equation}
  \label{eq:Xi_condition_2}
  \conj{\Omega}\Omega =c\cdot \pi_{\zeta^{d}}.  
\end{equation}
Equation~\eqref{eq:Xi_condition} is easy to solve for \(i\neq d/2\)
by choosing arbitrary invertible matrices~\(\Omega_{i,d-i}\) for
\(i\in\Z\) with \(i<d/2\) and letting \(\Omega_{d-i,i} \defeq c
\zeta^{d \cdot i} \conj{\Omega_{i,d-i}}^{-1}\).  If \(d\) is even and
\(i=d/2\), then~\eqref{eq:Xi_condition} becomes
\(\conj{\Omega_{d/2,d/2}} \cdot \Omega_{d/2,d/2} = c \cdot
\zeta^{d^2/2}\), which can be solved because \(c \cdot
\zeta^{d^2/2}\) is real if \(\conj{c}/c = \zeta^{d^2}\).  Thus there
is a large parameter space for our new braided quantum groups.
We now fix the deformation parameter~\(\zeta\).

A \emph{braided compact quantum group over~\(\T\)} is a unital
\(\Cst\)\nb-algebra~\(B\) with a continuous \(\T\)\nb-action
\(\beta\colon B\to B\otimes \Cont(\T)\) and a coassociative,
bisimplifiable comultiplication \(\Comult[B]\colon B\to
B\boxtimes_\zeta B\).  Here~\(\boxtimes_\zeta\) is a Rieffel
deformation of the usual spatial tensor product \(B\otimes B\),
where we use the canonical action of~\(\T^2\) on \(B\otimes B\).
We treat this using the monoidal structure on
\(\T\)\nb-\(\Cst\)\nb-algebras discussed
in~\cite{Kasprzak-Meyer-Roy-Woronowicz:Braided_SU2}*{Section~3}.  A
\emph{right representation} of~\(\G\) on~\((V,\pi)\) is a
\(\T\)\nb-invariant unitary \(u\in \U(\Endo(V)\otimes B)\) with
\(\Comult[B](u) = (j_1\otimes \Id)(u)\cdot (j_2\otimes \Id)(u)\),
where \(j_1, j_2\colon B\rightrightarrows B\boxtimes_\zeta B\) are
the two canonical embeddings.  There is a tensor product operation
on representations of~\(B\).
Roughly speaking, the braided orthogonal
quantum group \(A_o(V,\pi,\omega)\) is the universal \(\T\)\nb-braided
quantum group with a representation~\(u\) on~\((V,\pi)\) such that
\(\omega\in V\otimes V\) is invariant for the
representation~\(u\tenscorep u\).  This invariance condition is
equivalent to a linear relation among the matrix entries of \(u\)
and~\(u^*\).  The condition on~\(\omega\)
in~\eqref{eq:Xi_condition} ensures that this
relation is equivalent to its adjoint, so that it does not produce
linear relations among the matrix entries of~\(u\).
It also implies that the representation~\(u\) is irreducible.

We now describe the braided free orthogonal quantum groups
explicitly.  Choose a basis \(e_1,\dotsc,e_n\) for~\(V\) such that
\(\pi_z(e_i) = z^{d_i} e_i\) with \(d_1\le d_2 \le d_3 \le \dotsb
\le d_n\).  Write~\(\omega\)
in this basis as \(\omega = \sum_{i,j=1}^n \omega_{ij} e_i \otimes e_j\).
The \(\Cst\)\nb-algebra of the braided quantum group
\(A_o(V,\pi,\omega)\) is the universal unital \(\Cst\)\nb-algebra with
generators~\(u_{ij}\) for \(1\le i,j\le n\) subject to the relations
\begin{align}
  \label{eq:u_isometry}
  \sum_{k=1}^n u^*_{k i} u_{k j} &= \delta_{i,j},\\
  \label{eq:u_coisometry}
  \sum_{k=1}^n u_{i k} u^*_{j k} &= \delta_{i,j},\\
  \label{eq:xi_invariant}
  \zeta^{d_j d_i} \sum_{k=1}^n \omega_{i k} u_{j k}
  &= \zeta^{d_j (d-d_j)} \sum_{k=1}^n \omega_{k j} u_{k i}^*
\end{align}
for \(1\le i,j\le n\).
Here~\(u_{i j}\) is the \(i,j\)th coefficient of the canonical
representation \(u\in \Mat_n(A_o(V,\pi,\omega))\).  The two relations \eqref{eq:u_isometry}
and~\eqref{eq:u_coisometry} say that~\(u\) is unitary.
Equation~\eqref{eq:xi_invariant}
says that~\(\omega\) is
invariant.  The action~\(\beta\)
of~\(\T\) on~\(A_o(V,\pi,\omega)\) and the comultiplication
\(\Comult\colon A_o(V,\pi,\omega)\to A_o(V,\pi,\omega)\boxtimes_\zeta
A_o(V,\pi,\omega)\) are defined by
\begin{align}
  \label{eq:beta_on_u}
  \beta_z(u_{i k}) &= z^{d_k-d_i} u_{i k},\\
  \label{eq:comultiplication_u}
  \Comult(u_{i k}) &= \sum_{l=1}^n j_1(u_{i l}) j_2(u_{l k}),
\end{align}
where \(j_1,j_2\colon A_o(V,\pi,\omega)\rightrightarrows
A_o(V,\pi,\omega)\boxtimes_\zeta A_o(V,\pi,\omega)\) are the canonical
morphisms to the braided tensor product.  These are the unique
action of~\(\T\) and comultiplication for which~\(u\) is a
representation.

If \(V=\C^2\) and \(d_1=0\), \(d_2=1\), the definition above gives
the braided \(\qSU\) groups
of~\cite{Kasprzak-Meyer-Roy-Woronowicz:Braided_SU2}.
If \(d_i=0\) for \(i=1,\dotsc,n\), then~\(\beta\) is trivial and we get the usual
free orthogonal quantum groups of~\cite{Daele-Wang:Universal}.  As
an interesting new case, let \(n=2 m\) be even and let \(d_i=0\) for
\(1\le i\le m\) and \(d_i=1\) for \(m< i\le 2 m\).  Let \(d=1\).
Then a \(d\)\nb-homogeneous vector~\(\omega\) is equivalent to two
\(m\times m\)-matrices \(\Omega_0\), \(\Omega_1\), which form the
block matrix
\[
\begin{pmatrix}
  0&\Omega_1\\\Omega_0&0
\end{pmatrix}.
\]
Our construction requires two conditions, namely, \(\conj{\Omega_0}
\cdot \Omega_1 = c\) and \(\conj{\Omega_1} \cdot \Omega_0 = c\zeta\).
Thus~\(\Omega_0\) may be an arbitrary invertible \(m\times m\)-matrix,
and \(\Omega_1 = c \conj{\Omega_0}^{-1}\), \(\zeta = \conj{c}/c\).  If
\(m=1\), this gives the braided \(\qSU\) groups
of~\cite{Kasprzak-Meyer-Roy-Woronowicz:Braided_SU2}.  If~\(c\) is
real or, equivalently, if \(\zeta=1\), then it gives free orthogonal
quantum groups in the usual sense.

Banica~\cite{Banica:Rep_On} has described the
irreducible representations of the free orthogonal quantum groups
and their tensor products, showing that they behave like those of
the Lie group SU(2).  His result was made more precise by Bichon, De
Rijdt and Vaes~\cite{Bichon-de_Rijdt-Vaes:Ergodic}, who established
that any free orthogonal quantum group is monoidally equivalent to
\(\qSU\) for a specific~\(q\).  We prove analogues of these results
for our braided free orthogonal quantum groups.

\begin{theorem}
  \label{the:irrep_tensor}
  If~\(d\) is even, then the braided orthogonal quantum
  group~\(A_o(V,\pi,\omega)\) has irreducible
  representations~\(r_{(k,l)}\) for \(k\in\N\), \(l\in\Z\), such
  that any irreducible representation is unitarily equivalent to
  exactly one of these and \(\conj{r_{(k,l)}} = r_{(k,-l)}\) and
  \[
  r_{(a,b)} \otimes r_{(m,k)} \cong
  r_{(a+m,b+k)} \oplus r_{(a+m-2,b+k)} \oplus r_{(a+m-4,b+k)} \oplus \dotsb
  \oplus r_{(\abs{a-m},b+k)}.
  \]
  If~\(d\) is odd, then a similar statement holds, but we only allow
  those representations where \(a-b\) is even.
\end{theorem}

The fusion rules above are those of the group
\(\mathrm{SU}(2) \times \T\) in the even case, and those of the
group \(\mathrm{U}(2)\) in the odd case (see
\cite{Mrozinski:Qgrp_GL2_rep_type}*{Theorem 6.3}).  The extra circle
occurs because representations of a braided quantum group are
equivalent to representations of an associated quantum group,
namely, its semidirect product with the quantum group over which it
is braided (see \cite{Meyer-Roy:Braided_mu}*{Theorem~3.4}).  This is
a \(\Cst\)\nb-algebraic variant of the bosonisation of Radford and
Majid (see \cites{Radford:Hopf_projection, Majid:CrsdPrd_braidgrp}).
Therefore, we also change our notation from ``semidirect product''
to ``bosonisation''.  The classical case of bosonisation is a
semidirect product of groups.  The relevant classical example for us
is that \(\mathrm{U}(2)\) is a semidirect product of
\(\mathrm{SU}(2)\subseteq \mathrm{U}(2)\) and the circle
group~\(\T\), embedded into~\(\mathrm{U}(2)\) through
\(\iota\colon \T\to \mathrm{U}(2)\),
\(x\mapsto \textup{diag}(x,1)\).

The bosonisation for \(\qSU\) is isomorphic to \(\qU\) for the
quantum~\(\mathrm{U}(2)\) groups of \cites{Zhang:Uq2,
  Zhang-Zhao:Uq2}.  If~\(q\) is real, then \(\qU\) is a special case
of the groups \(\mathrm{U}_q(n)\), which go back
to~\cite{Noumi-Yamada-Mimachi:Rep_GLqn}.  If
\(q=\exp(\ima\vartheta)\) has absolute value~\(1\), then \(\qU\)
specialises to the \(\vartheta\)\nb-deformation
\(\mathrm{U}_\vartheta\) of~\(\mathrm{U}(2)\) defined by Connes and
Dubois-Violette~\cite{Connes-Dubois-Violette:NC_Spherical}.

In particular, when we view the usual \(\qSU\) for real~\(q\) as a
braided quantum group over~\(\T\), we are implicitly dealing with
the corresponding~\(\qU\), which is its bosonisation. 
We describe the representation category of
\(A_o(V,\pi,\omega)\) using the representation category of an ordinary
orthogonal quantum group.  Using the monoidal equivalences proven
in~\cite{Bichon-de_Rijdt-Vaes:Ergodic}, this implies a monoidal
equivalence between \(A_o(V,\pi,\omega)\) and either \(\qSU\times \T\)
or \(\qU\) depending on the parity of~\(d\):

\begin{theorem}
  \label{the:monoidal_equivalence}
  There is a unique \(q\in [-1,1]\setminus\{0\}\) such that the
  representation category of \(A_o(V,\pi,\omega)\) is equivalent to the
  representation category of\/ \(\qSU\times \T\) if~\(d\) is even and
  to the representation category of\/ \(\qU\) if~\(d\) is odd.
\end{theorem}

\section{The universal property}
\label{sec:construct}

Let \(\G=(B,\beta,\Comult)\) be a braided quantum group over~\(\T\).
A representation of~\(\G\) on~\(\C^n\) consists of a representation
of~\(\T\) on~\(\C^n\) and a unitary \(u\in \Mat_n(\C)\otimes B\)
that is \(\T\)\nb-invariant and satisfies a braided form of the
corepresentation condition (see, for instance,
\cites{Kasprzak-Meyer-Roy-Woronowicz:Braided_SU2,
  Meyer-Roy:Braided_mu}).  In particular, any representation
of~\(\T\) and \(u=1\) form a ``trivial'' representation of~\(\G\).
Let~\(\C_\ell\) for \(\ell\in\Z\) be~\(\C\) with the representation
\(z\mapsto z^\ell\) of~\(\T\) and with the representation \(u=1\)
of~\(\G\).  In the following, we will identify
\(\Mat_n(\C)\otimes B\) with \(\Mat_n(B)\) and write elements as
\((b_{ij})_{1\le i,j\le n}\) or
\(\sum_{1\le i,j\le n} e_{i j} \otimes b_{i j}\).

Let \(u=(u_{ij})_{1\le i,j\le n}\in\Mat_n(B)\)
be a representation of~\(\G\)
on \(V=\C^n\)
with underlying representation~\(\pi\)
of~\(\T\).
A \emph{\(d\)\nb-homogeneous
  \(\G\)\nb-invariant
  vector} is defined as a vector \(\omega\in V\otimes V\)
such that the map \(\C_d \to V\otimes V\),
\(c\mapsto c\cdot\omega\),
is an intertwining operator for the representations
of both \(\T\) and~\(\G\).
Let \(e_1,\dotsc,e_n\)
be an eigenbasis for~\(\pi\),
that is,
there are \(d_1,\dotsc,d_n\in\Z\)
with \(\pi_z(e_i) = z^{d_i} \cdot e_i\) for \(i=1,\dotsc,n\).
We order the basis so that
\(d_1\le d_2 \le d_3 \le \dotsb \le d_n\).
Let \(\omega\in V\otimes V\)
and write \(\omega = \sum_{i,j=1} \omega_{i,j} e_i \otimes e_j\).
The vector~\(\omega\)
is \(d\)\nb-homogeneous --
that is, the resulting map \(\C_d \to V\otimes V\)
is \(\T\)\nb-equivariant -- if and only if \(d_i + d_j = d\)
for all \(i,j\) with \(\omega_{i,j}\neq 0\).
We assume this from now on.

\begin{proposition}
  \label{pro:universal_property}
  The matrix \((u_{i j})_{1\le i,j\le n}\) satisfies the relations
  \eqref{eq:u_isometry}--\eqref{eq:comultiplication_u} if and only
  if it
  is a representation of~\(\G\) on~\(V\) with underlying
  representation~\(\pi\) of~\(\T\) and~\(\omega\)
  is a \(d\)\nb-homogeneous
  \(\G\)\nb-invariant vector in \(V\otimes V\).
\end{proposition}

\begin{proof}
  A representation must be unitary, \(\T\)\nb-invariant and
  satisfy
  \[
  \Comult(u) = (j_1\otimes \Id)(u) \cdot (j_2\otimes \Id)(u).
  \]
  Being unitary says that \(u^* u = 1\) and \(u u^* = 1\), which
  translates to the relations \eqref{eq:u_isometry}
  and~\eqref{eq:u_coisometry} for the coefficients~\(u_{i j}\).  The
  action~\(\Pi\) on~\(\Mat_n(\C)\) induced by~\(\pi\) is
  \(\Pi_z(e_{i j}) = z^{d_i - d_j} e_{i j}\).  Thus
  \(u = \sum e_{i j} \otimes u_{i j}\) is \(\T\)\nb-invariant if and
  only if the \(\T\)\nb-action~\(\beta\) on~\(B\) satisfies
  \(\beta_z(u_{i j}) = z^{d_j - d_i} u_{i j}\), which
  is~\eqref{eq:beta_on_u}.  Equation~\eqref{eq:comultiplication_u}
  is equivalent to
  \(\Comult(u) = (j_1\otimes \Id)(u)\cdot (j_2\otimes \Id)(u)\).
  Thus these four equations say exactly that~\(u\) is a
  representation of~\(\G\).  It remains to show that the
  \(\G\)\nb-invariance of~\(\omega\) is equivalent
  to~\eqref{eq:xi_invariant}.  We have already assumed that the map
  \(\C_d\to V\otimes V\), \(c\mapsto c\omega\), is
  \(\T\)\nb-equivariant.  We must describe when this map is
  \(\G\)\nb-equivariant for the tensor product representation
  of~\(\G\) on \(V\otimes V\).

  We represent~\(B\)
  faithfully on a Hilbert space to explain the definition of the
  tensor product for representations of braided quantum groups.
  Let~\(\Hils[L]\) be a separable Hilbert space with a continuous
  representation~\(\varrho\) of~\(\T\) and let \(B\hookrightarrow
  \Bound(\Hils[L])\) be a faithful, \(\T\)\nb-equivariant
  representation.  There is an orthonormal basis
  \((\lambda_m)_{m\in\N}\) for~\(\Hils[L]\) consisting of
  eigenvectors for the \(\T\)\nb-action, that is,
  \(\varrho_z(\lambda_m) = z^{l_m} \lambda_m\) with some \(l_m\in\N\).
  Since the representation of~\(B\) is \(\T\)\nb-equivariant,
  \begin{align*}
    \varrho_z(u_{ij} \lambda_m)
    &= \beta_z(u_{ij}) \varrho_z(\lambda_m)
    = z^{d_j-d_i + l_m} u_{ij} \lambda_m,\\
    \varrho_z(u^*_{ij} \lambda_m)
    &= \beta_z(u^*_{ij}) \varrho_z(\lambda_m)
    = z^{d_i-d_j + l_m} u_{ij}^* \lambda_m
  \end{align*}
  for all \(1\le i,j\le n\) and \(m\in\N\).

  Fix \(\zeta\in\T\) and define the \(\Rmat\)\nb-matrix on~\(\T\) by
  \(\Z\times\Z \ni (l,m)\to\zeta^{l m} \in\T\).  The associated
  braiding unitary \(\Braiding{\Hils[L]}{V}\colon
  \Hils[L]\otimes V \to V\otimes\Hils[L]\)
  and its inverse \(\Dualbraiding{V}{\Hils[L]}\colon
  V\otimes\Hils[L] \to \Hils[L]\otimes V\)
  act on the basis by
  \begin{equation}
    \label{eq:Braiding_L_V}
    \Braiding{\Hils[L]}{V}(\lambda_m\otimes e_i)
    = \zeta^{d_i\cdot l_m} e_i\otimes\lambda_m,
    \qquad
    \Dualbraiding{V}{\Hils[L]}(e_i \otimes \lambda_m)
    = \zeta^{-d_i\cdot l_m} \lambda_m\otimes e_i.
  \end{equation}
  The tensor product representation \(u\tenscorep u\) is the
  preimage in \(\Mat_{n^2}(B)\) of the unitary
  \begin{equation}
    \label{eq:tens_u}
    u\tenscorep u\defeq
    \Braiding{\Hils[L]}{V}_{23} u_{12}
    \Dualbraiding{V}{\Hils[L]}_{23} u_{23}
  \end{equation}
  on \(V\otimes V\otimes\Hils[L]\).  Thus the
  \(\G\)\nb-invariance of~\(\omega\) is equivalent to
  \[
  \Braiding{\Hils[L]}{V}_{23} u_{12}
  \Dualbraiding{V}{\Hils[L]}_{23} u_{23}
  \omega_{12} = \omega_{12}
  \]
  as operators \(\C_d \otimes \Hils[L] \to V\otimes V\otimes \Hils[L]\).
  This is equivalent to
  \begin{equation}
    \label{eq:equiv_invcond}
    u_{23}\omega_{12}
    = \bigl(\Braiding{\Hils[L]}{V}_{23} u_{12}
    \Dualbraiding{V}{\Hils[L]}_{23}\bigr)^* \omega_{12}
    = \Braiding{\Hils[L]}{V}_{23} u_{12}^*
    \Dualbraiding{V}{\Hils[L]}_{23} \omega_{12}.
  \end{equation}
  We use that \(u_{ki}^* \lambda_m\) is \(\T\)\nb-homogeneous of
  degree \(d_k-d_i+l_m\) to compute
  \begin{align*}
    \bigl(\Braiding{\Hils[L]}{V}_{23} u_{12}
    \Dualbraiding{V}{\Hils[L]}_{23}\bigr)^*
    e_k\otimes e_j\otimes\lambda_m
    &= \zeta^{-d_j l_m} \Braiding{\Hils[L]}{V}_{23}
    u^*(e_k\otimes\lambda_m)\otimes e_j\\
    &= \zeta^{-d_j l_m} \Braiding{\Hils[L]}{V}_{23}
    \biggl(\sum_{i=1}^n e_i\otimes u_{k i}^*(\lambda_m)\otimes e_j\biggr)\\
    &= \sum_{i=1}^n \zeta^{-d_j l_m+d_j(d_k-d_i+l_m)} e_i\otimes
    e_j\otimes u_{k i}^*(\lambda_m)\\
    &= \sum_{i=1}^n \zeta^{d_j(d_k-d_i)} e_i\otimes e_j\otimes
    u_{k i}^*(\lambda_m).
  \end{align*}
  As a result, the invariance condition~\eqref{eq:equiv_invcond}
  becomes
  \[
  \sum_{i,j,k=1}^n e_i\otimes e_j\otimes u_{jk} \omega_{ik}
  = \sum_{i,j,k=1}^n e_i\otimes e_j\otimes u_{ki}^*
  \zeta^{d_j(d_k-d_i)}\omega_{k j}.
  \]
  Comparing the coefficient of each summand \(e_i\otimes e_j\)
  gives~\eqref{eq:xi_invariant} because \(d_k + d_j = d\)
  whenever \(\omega_{k j}\neq0\)
  by the homogeneity of~\(\omega\).
\end{proof}

Assume the matrix \(\tilde{\omega}\defeq (\omega_{ij}\zeta^{d_i d_j})\)
to be invertible.  Multiplying both sides in~\eqref{eq:xi_invariant}
by~\((\tilde{\omega}^{-1})_{j l}\) and summing over~\(j\) gives
\begin{equation}
  \label{eq:xi_invariant2}
  \sum_{j,k=1}^n \zeta^{d_j d_i} u_{jk} \omega_{ik} (\tilde{\omega}^{-1})_{jl}
  = \sum_{j,k=1}^n \zeta^{d_j d_k} u_{ki}^* \omega_{kj} (\tilde{\omega}^{-1})_{jl}
  = \sum_{k=1}^n u_{ki}^* \delta_{k,l}
  = u_{l i}^*.
\end{equation}
Taking adjoints in~\(B\)
and substituting \((s,t)\) for \((j,k)\) then gives
\begin{equation}
  \label{eq:xi_invariant3}
  \sum_{s,t=1}^n \zeta^{-d_s d_i}
  \conj{\omega_{i t} (\tilde{\omega}^{-1})_{s l}} u^*_{s t}
  = u_{l i}.
\end{equation}
Now plug~\eqref{eq:xi_invariant2} for \(l=s\),
\(i=t\)
into~\eqref{eq:xi_invariant3}.  This gives the following linear
relation among the generators~\(u_{j k}\):
\begin{equation}
  \label{eq:lin_rel_u}
  u_{li} = \sum_{j,k,s,t=1}^n \zeta^{-d_s d_i + d_j d_t}
  \conj{\omega_{i t} (\tilde{\omega}^{-1})_{s l}} u_{j k} \omega_{t k}
  (\tilde{\omega}^{-1})_{j s}.
\end{equation}
We want the generators~\((u_{i j})\)
of our braided quantum group to be linearly independent.  So the
linear relation~\eqref{eq:lin_rel_u} should be trivial.
Equivalently,
\begin{equation}
  \label{eq:triv_rel_u1}
  \sum_{s,t=1}^n \zeta^{-d_s d_i + d_j d_t}
  \conj{\omega_{i t} (\tilde{\omega}^{-1})_{s l}} \omega_{t k}
  (\tilde{\omega}^{-1})_{j s}
  = \delta_{j,l} \delta_{i,k}.
\end{equation}

\begin{lemma}
  \label{lem:Xi_condition}
  The relation~\eqref{eq:triv_rel_u1} is equivalent
  to~\eqref{eq:Xi_condition}.
\end{lemma}

\begin{proof}
  Since
  \(\conj{\tilde{\omega}_{i t}} = \conj{\omega_{i t}} \zeta^{-d_i d_t}\)
  and \(\tilde{\omega}_{t k} = \omega_{t k} \zeta^{d_t d_k}\),
  \eqref{eq:triv_rel_u1} is equivalent to
  \[
  \delta_{j,l} \delta_{i,k} = \sum_{s,t=1}^n \zeta^{-d_s d_i + d_i d_t
    + d_j d_t - d_t d_k} \conj{\tilde{\omega}_{i t} (\tilde{\omega}^{-1})_{s
      l}} \tilde{\omega}_{t k} (\tilde{\omega}^{-1})_{j s}.
  \]
  Since \(\omega\)
  is of degree~\(d\) in \(V\otimes V\),
  so is~\(\tilde{\omega}\).
  Equivalently, the corresponding linear map
  \(\tilde{\omega}\colon V\to V\)
  maps elements of degree~\(a\)
  to elements of degree~\(d-a\).
  Then the inverse maps elements of degree~\(d-a\)
  to elements of degree~\(a\).
  Thus both \(\tilde{\omega}\)
  and~\(\tilde{\omega}^{-1}\)
  are homogeneous of degree~\(d\)
  in~\(V\otimes V\).
  That is, \(\tilde{\omega}_{a,b}=0\)
  and \(\tilde{\omega}^{-1}_{a,b}=0\)
  unless \(d_a+d_b=d\).
  Therefore, if
  \(\conj{\tilde{\omega}_{i t} (\tilde{\omega}^{-1})_{s l}} \tilde{\omega}_{t
    k} (\tilde{\omega}^{-1})_{j s} \neq0\),
  then \(d_i + d_t = d\),
  \(d_s+d_l=d\),
  \(d_t + d_k = d\)
  and \(d_j+d_s = d\).
  Equivalently, \(d_i = d_k\),
  \(d_l = d_j\),
  \(d_t = d-d_i\)
  and \(d_s = d - d_j\).  Then the exponent of~\(\zeta\) becomes
  \[
  -d_s d_i + d_i d_t + d_j d_t - d_t d_k = (d_j - d)d_i + d_j (d -
  d_i) = d (d_j - d_i).
  \]
  This does not involve the summation indices \(s,t\),
  so we may put it on the other side of the equation.
  Thus~\eqref{eq:triv_rel_u1} is equivalent to
  \[
  (\zeta^{-d\cdot d_j}\delta_{j,l}) \cdot (\zeta^{d\cdot d_i}
  \delta_{i,k}) = \sum_{s,t=1}^n \conj{\tilde{\omega}_{i t}
    (\tilde{\omega}^{-1})_{s l}} \tilde{\omega}_{t k} (\tilde{\omega}^{-1})_{j
    s} = (\tilde{\omega}^{-1} \conj{\tilde{\omega}^{-1}})_{j l}
  (\conj{\tilde{\omega}}\cdot \tilde{\omega})_{i k}.
  \]
  This is an equality of two exterior tensor products of matrices with
  the entries \(j,l\)
  and \(i,k\),
  respectively.  A matrix equality \(X_1\otimes Y_1 = X_2\otimes Y_2\)
  with \(X_1,Y_1\neq0\)
  holds if and only if \(X_1= c\cdot X_2\),
  \(Y_2= c\cdot Y_1\)
  with some non-zero scalar~\(c\).
  Thus we get a scalar~\(c\)
  with
  \((\conj{\tilde{\omega}}\cdot \tilde{\omega})_{i k} = c \zeta^{d d_i}
  \delta_{i,k}\)
  and
  \((\tilde{\omega}^{-1} \conj{\tilde{\omega}^{-1}})_{j l} = c^{-1} \zeta^{-
    d d_j} \delta_{j,l}\).
  These two conditions are equivalent to each other by taking
  inverses.

  Now we combine the entries~\(\omega_{i,j}\) with \(d_i=a\) and
  \(d_j=b\) into a block matrix~\(\Omega_{a,b}\) as
  in~\eqref{eq:Xi_condition}.  Recall that \(\Omega_{a,b}=0\) unless
  \(a+b=d\).  The change from \(\omega_{i,j}\) to~\(\tilde{\omega}_{i,j}\)
  replaces~\(\Omega_{a,d-a}\) by
  \(\tilde{\Omega}_{a,d-a} \defeq \zeta^{a\cdot (d-a)}\Omega_{a,d-a}\).
  So
  \((\conj{\tilde{\omega}}\cdot \tilde{\omega})_{i k} = c \zeta^{d d_i}
  \delta_{i,k}\) is equivalent to
  \(\zeta^{a\cdot (a-d)}\conj{\Omega_{a,d-a}} \cdot \zeta^{a\cdot
    (d-a)}\Omega_{d-a,a} = c\cdot \zeta^{d a}\cdot 1\), where~\(1\)
  denotes the identity matrix of appropriate size.  This is
  exactly~\eqref{eq:Xi_condition}.
\end{proof}

Lemma~\ref{lem:Xi_condition} explains why we impose the
condition~\eqref{eq:Xi_condition} for our construction: otherwise,
some of our generators~\(u_{i j}\) would be redundant.  The
following lemma was suggested by the anonymous referee.  It
interprets this relation in another way:

\begin{lemma}
  \label{lem:irr_fundrep}
  Assume only that~\(\Omega\) is invertible.  The
  representation~\(u\) is irreducible if and only
  if~\eqref{eq:Xi_condition} holds.
\end{lemma}  

\begin{proof}
  The proof of Theorem~\ref{the:irrep_tensor} will show that~\(u\)
  is irreducible.  It uses~\eqref{eq:Xi_condition}, which is a
  standing assumption for the construction of \(A_o(V,\pi,\omega)\).
  Here we will only prove the converse.  If~\(u\) is irreducible,
  any intertwiner of~\(u\) is a constant multiple of the identity.
  We are going to rewrite~\eqref{eq:lin_rel_u} as saying that the
  transpose of \(\conj{\Omega}\Omega \pi_{\zeta^{-d}}\) is an
  intertwiner of~\(u\).

  We have seen in the proof of Lemma~\ref{lem:Xi_condition} that
  \(\omega\) and~\(\tilde{\omega}\) are homogeneous of degree~\(d\).
  Therefore, any non-zero summand on the right hand side
  of~\eqref{eq:lin_rel_u} has \(d_i+d_t=d = d_i + d_k\) and
  \(d_j + d_s = d\).  Thus we may replace the exponent
  \(d_j d_t - d_s d_i\) of~\(\zeta\) by
  \(d_j (d-d_i) - d_i (d-d_j)= d d_j - d d_k\).  We combine the
  entries \(\omega_{i j}\) and~\(\tilde{\omega}_{i j}\) to matrices
  \(\Omega\) and~\(\tilde{\Omega}\), and we let~\(u^{\top}\) be the
  transpose of~\(u\).  We rewrite~\eqref{eq:lin_rel_u} as an
  equality of matrix products:
  \[
    u^{\top} =  \conj{\Omega} \Omega \pi_{\zeta^{-d}} u^{\top} \pi_{\zeta^{d}}
    \tilde{\Omega}^{-1} \conj{\tilde{\Omega}}^{-1}.
  \]
  This is equivalent to
  \(u^{\top}\conj{\tilde{\Omega}} \tilde{\Omega} \pi_{\zeta^{-d}} =
  \conj{\Omega} \Omega \pi_{\zeta^{-d}} u^{\top}\).  Using
  that~\(\Omega\) is homogeneous, we rewrite the entries of
  \(\conj{\tilde{\Omega}} \tilde{\Omega}\) as
  \[
    (\conj{\tilde{\Omega}}\tilde{\Omega})_{a c}
    = \sum_{b=1}^n \zeta^{- d_a d_b + d_b d_c}
    \omega_{a b}\omega_{b c}
    = \sum_{b=1}^n \zeta^{-(d-d_b) d_b + d_b (d-d_b)}
    \omega_{a b}\omega_{b c}
    = \sum_{b=1}^n \omega_{a b}\omega_{b c}.
  \]
  That is,
  \(\conj{\tilde{\Omega}} \tilde{\Omega} = \conj{\Omega} \Omega\).
  Thus
  \(\conj{\Omega}\Omega\pi_{\zeta^{-d}} u^{\top} = \conj{\Omega}
  \Omega \pi_{\zeta^{-d}} u^{\top}\).  Taking transposes, we see
  that the transpose of \(\conj{\Omega}\Omega \pi_{\zeta^{-d}}\) is
  an intertwiner for the representation~\(u\).  Thus it is a scalar
  multiple of~\(1\).  Equivalently,
  \(\conj{\Omega}\Omega \pi_{\zeta^{-d}} = c\cdot 1_n\) for some
  \(c\in\C\).  Since~\(\Omega\) is invertible, \(c\in \C^\times\).
  This is equivalent to~\eqref{eq:Xi_condition_2}, which is
  equivalent to~\eqref{eq:Xi_condition}.
\end{proof}

So far we have worked with a general braided quantum group~\(\G\) to
explain the defining equations
\eqref{eq:u_isometry}--\eqref{eq:comultiplication_u} of
\(A_o(V,\pi,\omega)\).
Now we show that these equations do define a braided quantum group.

Since the matrix
\((u_{i j})_{1\le i,j\le n}\in \Mat_n(A_o(V,\pi,\omega))\) is unitary,
\(\norm{u_{i j}}\le 1\) for \(1\le i,j\le n\).  Hence the universal
\(\Cst\)\nb-algebra with the generators~\(u_{ij}\) and the relations
\eqref{eq:u_isometry}--\eqref{eq:xi_invariant} exists.  To construct
it, we start with the universal unital \Star{}algebra
\(\mathcal{A}_o(V,\pi,\omega)\) with these generators and relations.
Since any \(\Cst\)\nb-seminorm on \(\mathcal{A}_o(V,\pi,\omega)\)
satisfies \(\norm{u_{i j}}\le 1\) for \(1\le i,j\le n\), there is a
largest \(\Cst\)\nb-seminorm on \(\mathcal{A}_o(V,\pi,\omega)\).  And
\(A_o(V,\pi,\omega)\) is the completion of \(\mathcal{A}_o(V,\pi,\omega)\)
in this largest \(\Cst\)\nb-seminorm.

\begin{proposition}
  \label{pro:beta_coaction}
  There is a unique unital \Star{}homomorphism
  \(\beta\colon A_o(V,\pi,\omega)\to \Cont(\T,A_o(V,\pi,\omega))\)
  satisfying~\eqref{eq:beta_on_u}.  It defines a continuous
  \(\T\)\nb-action~\(\beta\) on~\(A_o(V,\pi,\omega)\).
\end{proposition}

\begin{proof}
  The formula \(\beta_z(u_{i j}) = z^{d_j - d_i} u_{i j}\)
  clearly defines a \(\T\)\nb-action
  on the free \Star{}algebra
  generated by~\(u_{i,j}\).
  We must show that the relations that define~\(A_o(V,\pi,\omega)\)
  are homogeneous.  The relations \eqref{eq:u_isometry}
  and~\eqref{eq:u_coisometry} are homogeneous because
  \[
  \beta_z(u_{k i}^* u_{k j}) = z^{d_j - d_i} u_{k i}^* u_{k j},\qquad
  \beta_z(u_{i k} u_{j k}^*) = z^{d_j - d_i} u_{i k} u_{j k}^*
  \]
  for all \(1\le i,j,k\le n\).
  The relation~\eqref{eq:xi_invariant} is homogeneous because
  \(\omega_{i,k}=0\) for \(d_i+d_k\neq d\), so that
  \begin{align*}
    \beta_z(\zeta^{d_i d_j} \omega_{i k} u_{j k})
    &= z^{d_k - d_j} \cdot \zeta^{d_i d_j} \omega_{i k} u_{j k}
    = z^{d- d_j - d_i} \cdot \zeta^{d_i d_j} \omega_{i k} u_{j k},\\
    \beta_z(\zeta^{d_j d_k} \omega_{k j} u_{k i}^*)
    &= z^{d_k - d_i} \cdot \zeta^{d_j d_k} \omega_{k j} u_{k i}^*
    = z^{d - d_i - d_j} \cdot \zeta^{d_j d_k} \omega_{k j} u_{k i}^*
  \end{align*}
  for all \(i,j,k\).
\end{proof}

\begin{proposition}
  \label{pro:comult_exists}
  There is a unique unital \Star{}homomorphism \(\Comult\colon
  A_o(V,\pi,\omega)\to A_o(V,\pi,\omega)\boxtimes_\zeta A_o(V,\pi,\omega)\)
  satisfying~\eqref{eq:comultiplication_u}.  It is
  \(\T\)\nb-equivariant, coassociative, and satisfies the Podleś
  condition.
\end{proposition}

\begin{proof}
  We abbreviate \(B\defeq A_o(V,\pi,\omega)\).  Recall the faithful
  representation \(B\hookrightarrow \Bound(\Hils[L])\) from the
  proof of Proposition~\ref{pro:universal_property}.  Define the
  braiding unitaries \(\Braiding{\Hils[L]}{\Hils[L]}\) and
  \(\Dualbraiding{\Hils[L]}{\Hils[L]}\) on
  \(\Hils[L] \otimes \Hils[L]\) by
  \begin{equation}
    \label{eq:Braiding_L_L}
    \Braiding{\Hils[L]}{\Hils[L]}(\lambda_a \otimes \lambda_b)
    \defeq \zeta^{l_a l_b} \lambda_b \otimes \lambda_a,\qquad
    \Dualbraiding{\Hils[L]}{\Hils[L]}(\lambda_a \otimes \lambda_b)
    \defeq \zeta^{-l_a l_b} \lambda_b \otimes \lambda_a.
  \end{equation}
  There is a faithful representation
  \(B\boxtimes B\hookrightarrow \Bound(\Hils[L]\otimes\Hils[L])\),
  such that the inclusions
  \(j_1,j_2\colon B \rightrightarrows B\boxtimes B\) of the tensor
  factors become \(j_1(x) = x_1\) and
  \(j_2(x) = \Braiding{\Hils[L]}{\Hils[L]}
  x_1\Dualbraiding{\Hils[L]}{\Hils[L]}\) in leg numbering notation.
  Let
  \(U\defeq u_{12}\Braiding{\Hils[L]}{\Hils[L]}_{23}
  u_{12}\Dualbraiding{\Hils[L]}{\Hils[L]}_{23} \in \Bound(V\otimes
  \Hils[L]\otimes \Hils[L])\).  A \Star{}homomorphism
  \(\Comult\colon B\to B\boxtimes B\)
  satisfying~\eqref{eq:comultiplication_u} exists if and only if the
  matrix coefficients of~\(U\) satisfy the relations
  \eqref{eq:u_isometry}--\eqref{eq:xi_invariant} that define
  \(A_o(V,\pi,\omega)\).  Equations \eqref{eq:u_isometry}
  and~\eqref{eq:u_coisometry} hold because~\(U\) is unitary.  Define
  \[
    U\tenscorep U \defeq
    \Braiding{\Hils[L] \otimes \Hils[L]}{V}_{234}
    U_{123}
    \Dualbraiding{V}{\Hils[L] \otimes \Hils[L]}_{234}
    U_{234}
    \in \Bound(V\otimes V\otimes \Hils[L] \otimes \Hils[L])
  \]
  as in~\eqref{eq:tens_u}; here
  \begin{equation}
    \label{eq:Braiding_LL_V}
    \Braiding{\Hils[L] \otimes \Hils[L]}{V}
    (\lambda_a \otimes \lambda_b \otimes e_i)
    = \zeta^{(l_a+l_b)d_i} e_i \otimes \lambda_a \otimes \lambda_b,
    \qquad
    \Dualbraiding{\Hils[L] \otimes \Hils[L]}{V} =
    (\Braiding{\Hils[L] \otimes \Hils[L]}{V})^*.
  \end{equation}
  The proof of Proposition~\ref{pro:universal_property} implies
  \(\Braiding{\Hils[L]}{V}_{23} u_{12}
  \Dualbraiding{V}{\Hils[L]}_{23} u_{23}\omega_{12} = (u\tenscorep
  u)\omega_{12}=\omega_{12}\) and shows that the coefficients of~\(U\)
  satisfy~\eqref{eq:xi_invariant} if and only if
  \((U\tenscorep U) \omega_{12} = \omega_{12}\).  Now we compute
  \[
    (U\tenscorep U)\omega_{12}
    = \Braiding{\Hils[L] \otimes \Hils[L]}{V}_{234}
    (u_{12}\Braiding{\Hils[L]}{\Hils[L]}_{23}
    u_{12}\Dualbraiding{\Hils[L]}{\Hils[L]}_{23})
    \Dualbraiding{V}{\Hils[L] \otimes \Hils[L]}_{234}
    (u_{23}\Braiding{\Hils[L]}{\Hils[L]}_{34}
    u_{23}\Dualbraiding{\Hils[L]}{\Hils[L]}_{34})\omega_{12}.
  \]
  We are going to prove \((U\tenscorep U)\omega_{12} = \omega_{12}\) using
  \((u\tenscorep u)\omega_{12}=\omega_{12}\) and properties of the
  braiding unitaries.  This computation is more readable when
  written in pictures.  Using the definition~\eqref{eq:tens_u} of
  \(u \tenscorep u\), we rewrite
  \((u\tenscorep u)\omega_{12}=\omega_{12}\) through the diagram
  \begin{equation}
    \label{eq:inv_vect}
    \begin{tikzpicture}[baseline=(current bounding box.west),scale=.4]
      \draw (1,7.5)--(1,7.2);
      \draw (0.4,6)--(1.6,6)--(1.6,7.2)--(0.4,7.2)--(0.4,6);
      \draw (1,6.6) node {$\omega$};
      \draw (0.5,6)--(0.5,2.5);
      \draw (1.5,6)--(1.5,5);
      \draw (2.5,7.5)--(2.5,5);
      \draw (1.4,5)--(2.6,5)--(2.6,3.8)--(1.4,3.8)--(1.4,5);
      \draw (2,4.4) node {$u$};
      \draw (1.6,3.8) to[out=315, in=90] (2.5,1.9) to[out=270, in=45] (1.6,0);
      \draw[overar] (2.4,3.8)--(1.5,2.5);
      \draw (0.4,2.5)--(1.6,2.5)--(1.6,1.3)--(0.4,1.3)--(0.4,2.5);
      \draw (1,1.9) node {$u$};
      \draw (0.5,1.3)--(0.5,0);
      \draw[overar] (1.5,1.3)--(2.5,0);
    \end{tikzpicture}
    \ =\
    \begin{tikzpicture}[baseline=(current bounding box.west),scale=.4]
      \draw (1,8)--(1,7.2);
      \draw (0.4,6)--(1.6,6)--(1.6,7.2)--(0.4,7.2)--(0.4,6);
      \draw (1,6.6) node {$\omega$};
      \draw (0.5,6)--(0.5,5);
      \draw (1.5,6)--(1.5,5);
      \draw (2.5,8)--(2.5,5);
    \end{tikzpicture}
  \end{equation}
  The following diagrams illustrate the proof that
  \((U\tenscorep U)\omega_{12} =\omega_{12}\).
  
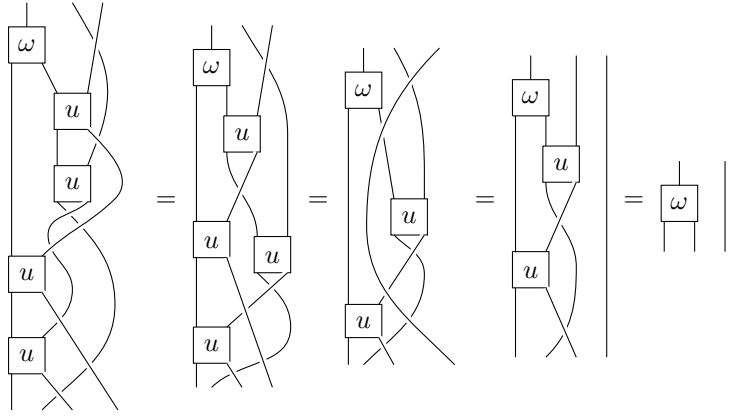
\[
  \begin{tikzpicture}[baseline=(current bounding box.west),scale=.4]
    \draw (1.5,8)--(1.5,7.2);
    \draw (0.9,6)--(2.1,6)--(2.1,7.2)--(0.9,7.2)--(0.9,6);
    \draw (1.5,6.6) node {$\omega$};
    \draw (1,6)--(1,-0.4);
    \draw (2,6)--(2.5,5);
    \draw (3,8) to[out=300, in=80] (4.1,4.4) to[out=260, in=70] (3.5,2.6);
    \draw[overar] (4,8)--(3.5,5);
    \draw (2.4,5)--(3.6,5)--(3.6,3.8)--(2.4,3.8)--(2.4,5);
    \draw (3,4.4) node {$u$};
    \draw (2.5,3.8)--(2.5,2.6);
    \draw (2.4,2.6)--(3.6,2.6)--(3.6,1.4)--(2.4,1.4)--(2.4,2.6);
    \draw (3, 2) node {$u$};
    \draw (2.5,1.4) to[out=315, in=90] (4.4,-2) to[out=270, in=45] (2,-5.5);
    \draw[overar] (3.5,1.4) to[out=225, in=90] (2.2, 0.3) to[out=270, in=90] (3,-1.5) to[out=270, in=45] (2,-3.1);
    \draw[overar] (3.5,3.8) to[out=315, in=45] (4.4,1.5) to[out=225, in=45] (2,-0.4);
    \draw (0.9,-0.4)--(2.1,-0.4)--(2.1,-1.6)--(0.9,-1.6)--(0.9,-0.4);
    \draw (1.5,-1) node {$u$};
    \draw[overar] (2,-1.6)--(4.5,-5.5);
    \draw (0.9,-3.1)--(2.1,-3.1)--(2.1,-4.3)--(0.9,-4.3)--(0.9,-3.1);
    \draw (1.5,-3.7) node {$u$};
    \draw (1,-1.6)--(1,-3.1);
    \draw (1,-4.3)--(1,-5.5);
    \draw[overar] (2,-4.3)--(3,-5.5);
  \end{tikzpicture}
  \ =\
  \begin{tikzpicture}[baseline=(current bounding box.west),scale=.4]
    \draw (2,8)--(2,7.2);
    \draw (1.4,6)--(2.6,6)--(2.6,7.2)--(1.4,7.2)--(1.4,6);
    \draw (2,6.6) node {$\omega$};
    \draw (1.5,6)--(1.5,1.5);
    \draw (2.5,6)--(2.5,5);
    \draw (3,8) to[out=300, in=90] (4.5,4.4) to[out=270, in=90] (4.5,1);
    \draw[overar] (4,8)--(3.5,5);
    \draw (2.4,5)--(3.6,5)--(3.6,3.8)--(2.4,3.8)--(2.4,5);
    \draw (3,4.4) node {$u$};
    \draw (2.5,3.8) to[out=270, in=90] (3.5,1);
    \draw[overar] (3.5,3.8)--(2.5,1.5);
    \draw (1.4,1.5)--(2.6,1.5)--(2.6,0.3)--(1.4,0.3)--(1.4,1.5);
    \draw (2,0.8) node {$u$};
    \draw (1.5,0.3)--(1.5,-2);
    \draw (3.4,1)--(4.6,1)--(4.6,-0.2)--(3.4,-0.2)--(3.4,1);
    \draw (4,0.3) node {$u$};
    \draw (3.5,-0.2) to[out=315 , in=90] (4.6,-2) to[out=270, in=45] (2,-4);
    \draw[overar] (4.5,-0.2)--(2.5,-2);
    \draw (1.4,-2)--(2.6,-2)--(2.6,-3.2)--(1.4,-3.2)--(1.4,-2);
    \draw (2,-2.6) node {$u$};
    \draw (1.5,-3.2)--(1.5,-4);
    \draw[overar] (2.5,-3.2)--(3,-4);
    \draw[overar] (2.5,0.3)--(4,-4);
  \end{tikzpicture}
  \ =\
  \begin{tikzpicture}[baseline=(current bounding box.west),scale=.4]
    \draw (2,8.5)--(2,7.7);
    \draw (1.4,6.5)--(2.6,6.5)--(2.6,7.7)--(1.4,7.7)--(1.4,6.5);
    \draw (2,7.1) node {$\omega$};
    \draw (1.5,6.5)--(1.5,0);
    \draw (2.5,6.5)--(3,3.5);
    \draw (3,8.5) to[out=300, in=90] (4,4.4) to[out=270, in=90] (4,3.5);
    \draw (2.9,3.5)--(4.1,3.5)--(4.1,2.3)--(2.9,2.3)--(2.9,3.5);
    \draw (3.5,2.9) node {$u$};
    \draw (3,2.3) to[out=315, in=90] (4,1) to[out=270, in=45] (2,-2);
    \draw[overar] (4,2.3)--(2.5,0);
    \draw (1.4,0)--(2.6,0)--(2.6,-1.1)--(1.4,-1.1)--(1.4,0);
    \draw (2,-0.6) node {$u$};
    \draw (1.5,-1.1)--(1.5,-2);
    \draw[overar] (2.5,-1.1)--(3,-2);
    \draw[overar] (4.5,8.5) to[out=225, in=90](2.1,2.6) to[out=270, in=135] (5,-2);
  \end{tikzpicture}
  \ =\
  \begin{tikzpicture}[baseline=(current bounding box.west),scale=.4]
    \draw (2,8)--(2,7.2);
    \draw (1.4,6)--(2.6,6)--(2.6,7.2)--(1.4,7.2)--(1.4,6);
    \draw (2,6.6) node {$\omega$};
    \draw (1.5,6)--(1.5,1.5);
    \draw (2.5,6)--(2.5,5);
    \draw (4.5,8)--(4.5,-2);
    \draw[overar] (3.5,8)--(3.5,5);
    \draw (2.4,5)--(3.6,5)--(3.6,3.8)--(2.4,3.8)--(2.4,5);
    \draw (3,4.4) node {$u$};
    \draw (2.5,3.8) to[out=270, in=90] (3.5,1) to[out=270, in=45] (2.5,-2);
    \draw[overar] (3.5,3.8)--(2.5,1.5);
    \draw (1.4,1.5)--(2.6,1.5)--(2.6,0.3)--(1.4,0.3)--(1.4,1.5);
    \draw (2,0.8) node {$u$};
    \draw (1.5,0.3)--(1.5,-2);
    \draw[overar] (2.5,0.3)--(3.5,-2);
  \end{tikzpicture}
  \ =\
  \begin{tikzpicture}[baseline=(current bounding box.west),scale=.4]
    \draw (1,8)--(1,7.2);
    \draw (0.4,6)--(1.6,6)--(1.6,7.2)--(0.4,7.2)--(0.4,6);
    \draw (1,6.6) node {$\omega$};
    \draw (0.5,6)--(0.5,5);
    \draw (1.5,6)--(1.5,5);
    \draw (2.5,8)--(2.5,5);
    \draw (3.5,8)--(3.5,5);
  \end{tikzpicture}
  \]
  This finishes the construction of
  \(\Comult\colon A_o(V,\pi,\omega) \to A_o(V,\pi,\omega)\boxtimes
  A_o(V,\pi,\omega)\).

  A direct computation shows that~\(\Comult\) is
  \(\T\)\nb-equivariant for the \(\T\)\nb-action~\(\beta\)
  on \(A_o(V,\pi,\omega)\) and the induced \(\T\)\nb-action on
  \(A_o(V,\pi,\omega)\boxtimes A_o(V,\pi,\omega)\).  It is coassociative
  because both \((\Comult\boxtimes\Id)\circ\Comult\) and
  \((\Id\boxtimes\Comult)\circ\Comult\) map~\(u\) to
  \(j_1(u)\cdot j_2(u)\cdot j_3(u)\).

  The Podle\'s condition for \(B\defeq A_o(V,\pi,\omega)\) follows by
  the same argument as
  in~\cite{Kasprzak-Meyer-Roy-Woronowicz:Braided_SU2}*{Section~4}.
  Let \(S\defeq \setgiven{b\in B}{j_1(b) \in \Comult(B)j_2(B)}\).
  The set~\(S\) contains \(u_{ij}\) and~\(u_{ij}^*\) for all
  \(1\le i,j\le n\) because~\eqref{eq:comultiplication_u} is
  equivalent to
  \(j_1(u_{ik}) = \sum_{j=1}^n \Comult(u_{ij}) j_2(u_{kj}^*)\).  Let
  \(x,y\in S\) be homogeneous of degree \(k\) and~\(l\),
  respectively.  Then \(j_1(x) j_2(y) = \zeta^{kl}j_2(y) j_1(x)\).
  This implies \(j_1(x) j_2(B) = j_2(B) j_1(x)\) and
  \(j_1(B) j_2(y) = j_2(y) j_1(B)\) as in
  \cite{Kasprzak-Meyer-Roy-Woronowicz:Braided_SU2}*{Proposition~3.1}.
  So
  \begin{align*}
    j_1(x\cdot y)
    = j_1(x) j_1(y) \in \Comult(B) j_2(B)j_1(y)
    &= \Comult(B)j_1(y)j_2(B)
    \\ &\subseteq \Comult(B)\Comult(B) j_2(B) j_2(B)
         = \Comult(B)j_2(B).
  \end{align*}
  Thus \(x\cdot y\in S\).
  Therefore, all the monomials in \(u_{ij}\) and~\(u_{ij}^*\)
  belong to~\(S\).  Then~\(S\) is dense in~\(B\).  This gives one
  of the Podle\'s conditions:
  \(B\boxtimes B = j_1(B)j_2(B) \subseteq \Comult(B)j_2(B)j_2(B) =
  \Comult(B)j_2(B)\).  The other Podle\'s condition is shown
  similarly by proving that
  \(R\defeq \setgiven{x\in B}{j_2(x)\in j_1(B) \Comult(B)}\) is
  dense in~\(B\).
\end{proof}

Propositions \ref{pro:beta_coaction} and~\ref{pro:comult_exists}
show that \(A_o(V,\pi,\omega)\) with the \(\T\)\nb-action~\(\beta\) and
the comultiplication~\(\Comult\) is a braided compact quantum group
over~\(\T\).  By construction,
\(u\in \Mat_n(A_o(V,\pi,\omega))\) and the underlying
representation~\(\pi\) of~\(\T\) form a representation on
\(V\cong\C^n\).  Up to isomorphism, \(A_o(V,\pi,\omega)\) does not
depend on the choice of the basis \(e_1,\dotsc,e_n\) in~\(V\).  To
make this clearer, we now treat~\(u\) as a representation in
\(\Bound(V) \otimes A_o(V,\pi,\omega)\).

\begin{theorem}
  \label{the:universal_Ao}
  The braided quantum group \(A_o(V,\pi,\omega)\)
  is the universal one with a representation on~\((V,\pi)\)
  for which \(\omega\in V\otimes V\)
  is \(d\)\nb-homogeneous
  and \(\G\)\nb-invariant.
  That is, if\/ \(\G=(B,\beta,\Comult)\)
  is another braided quantum group and \(U\in \Bound(V) \otimes B\)
  is a representation of\/~\(\G\)
  on~\((V,\pi)\)
  such that \(\omega\in V\otimes V\) is \(d\)\nb-homogeneous
  and \(\G\)\nb-invariant,
  then there is a unique \(\T\)\nb-equivariant
  Hopf \Star{}homomorphism \(A_o(V,\pi,\omega) \to B\)
  mapping \(u\mapsto U\).
\end{theorem}

\begin{proof}
  Choose any orthonormal basis \(f_1,\dotsc,f_n\) of
  \(\T\)\nb-eigenvectors in~\(V\) as above and use it to turn~\(U\)
  into a matrix \((U_{ij})_{1\le i,j\le n} \in \Mat_n(B)\).
  Proposition~\ref{pro:universal_property} gives a unique unital
  \Star{}homomorphism \(A_o(V,\pi,\omega)\to B\) that maps
  \(u_{i j}\mapsto U_{i j}\) and that is \(\T\)\nb-equivariant and
  compatible with comultiplications.
\end{proof}

\begin{corollary}
  Up to isomorphism, the braided quantum group \(A_o(V,\pi,\omega)\)
  does not depend on the choice of the basis \(e_1,\dotsc,e_n\) used
  to build it.  Even more, a \(\T\)\nb-equivariant unitary operator
  \(\varphi\colon (V_1,\pi_1) \to (V_2,\pi_2)\) with
  \(\varphi(\omega_1) = \omega_2\) induces an isomorphism of braided
  quantum groups
  \(A_o(V_1,\pi_1,\omega_1) \cong A_o(V_2,\pi_2,\omega_2)\).
\end{corollary}

\begin{proof}
  The universal property in Theorem~\ref{the:universal_Ao} does not
  mention the basis \(e_1,\dotsc,e_n\).  Hence the braided quantum
  groups built from different orthonormal bases of eigenvectors for
  the \(\T\)\nb-action~\(\pi\) have the same universal property.
  This implies that they are isomorphic as braided quantum groups.
  More generally, let \(\varphi\colon (V_1,\pi_1) \to (V_2,\pi_2)\)
  be a \(\T\)\nb-equivariant unitary operator with
  \(\varphi(\omega_1) = \omega_2\).  Let \(\G=(B,\beta,\Comult)\) be
  another braided quantum group and let
  \(U_1\in \Bound(V_1) \otimes B\) be a representation of\/~\(\G\)
  on~\((V_1,\pi_1)\) such that \(\omega_1\in V_1\otimes V_1\) is
  \(d\)\nb-homogeneous and \(\G\)\nb-invariant.  Then
  \(U_2 \defeq (\varphi\otimes 1) U_1 (\varphi\otimes 1)^* \in
  \Bound(V_2) \otimes B\) is a representation of~\(\G\) on
  \((V_2,\pi_2)\) such that \(\omega_2\in V_2\otimes V_2\) is
  \(d\)\nb-homogeneous and \(\G\)\nb-invariant.  Conversely, any
  representation~\(U_2\) of~\(\G\) on \((V_2,\pi_2)\) such that
  \(\omega_2\in V_2\otimes V_2\) is \(d\)\nb-homogeneous and
  \(\G\)\nb-invariant gives a representation
  \(U_1 \defeq (\varphi\otimes 1)^* U_2 (\varphi\otimes 1) \in
  \Bound(V_1) \otimes B\) on \((V_1,\pi_1)\) such that
  \(\omega_1\in V_1\otimes V_1\) is \(d\)\nb-homogeneous and
  \(\G\)\nb-invariant.  By Theorem~\ref{the:universal_Ao}, this
  bijection on certain classes of representations induces a
  bijection between \(\T\)\nb-equivariant Hopf \Star{}homomorphisms
  from \(A_o(V_1,\pi_1,\omega_1)\) and \(A_o(V_2,\pi_2,\omega_2)\)
  to~\(\G\).  By the Yoneda Lemma, this bijection comes from an
  isomorphism of braided quantum groups
  \(A_o(V_1,\pi_1,\omega_1) \cong A_o(V_2,\pi_2,\omega_2)\).
\end{proof}

Next we describe some isomorphisms of our braided quantum groups
that shift the homogeneity degree~\(d\) of~\(\omega\).  Let \((V,\pi)\)
be a representation of~\(\T\) on a finite-dimensional vector space
and let \(s\in\Z\).  Then \(\pi'_z \defeq z^{-s}\cdot \pi_z\) is
another representation of~\(\T\) on~\(V\).  Let
\(d_1,\dotsc,d_n\in\Z\) and let \(e_1,\dotsc,e_n\in V\) be an
orthonormal basis of~\(V\) such that \(\pi_z(e_i) = z^{d_i} e_i\)
for all \(z\in\T\).  Then \(\pi'_z(e_i) = z^{d_i'} e_i\) with
\(d_i' \defeq d_i - s\), that is, the degrees~\(d_i\) are shifted.

\begin{lemma}
  \label{lem:shift_degrees}
  Let
  \[
  \omega' = \sum_{i,j=1}^n \zeta^{-s\cdot d_j} \omega_{i j} e_i \otimes e_j.
  \]
  Then there is a \(\T\)\nb-equivariant Hopf \Star{}isomorphism
  \(A_o(V,\pi,\omega) \cong A_o(V,\pi',\omega')\) which maps the
  generator~\(u_{ij}\) of~\(A_o(V,\pi,\omega)\) to
  \(u_{ij}\in A_o(V,\pi',\omega')\).
\end{lemma}

\begin{proof}
  When we replace \(\pi,\omega\) by \(\pi',\omega'\), then we
  replace~\(\omega_{i k}\) by \(\zeta^{-s d_k} \omega_{i k}\), \(d_i\) by
  \(d_i'= d_i-s\), and~\(d\) by \(d'=d-2 s\) because~\(\omega'\) has
  degree \(d-2 s\) for the representation~\(\pi'\).  The block
  matrix~\(\Omega\) in the introduction becomes
  \(\Omega'_{i j} = \zeta^{-s\cdot j}\cdot\Omega_{i j}\).
  Equation~\eqref{eq:Xi_condition} for~\(\Omega\) implies
  \[
    \conj{\Omega'_{i,d-i}} \cdot \Omega'_{d-i,i}
    = \zeta^{s (d-i) -i} \conj{\Omega_{i,d-i}} \cdot \Omega_{d-i,i}
    = \zeta^{s d - 2 s i}\cdot c\cdot \zeta^{d\cdot i}\cdot 1
    = c'\cdot \zeta^{d'\cdot i}\cdot 1
  \]
  with \(c' \defeq \zeta^{s d} c\).  Thus \(A_o(V,\pi',\omega')\) is
  defined.

  The effect of our substitutions on~\eqref{eq:xi_invariant} is that
  the two sides \(\zeta^{d_j d_i} \sum_{k=1}^n \omega_{i k} u_{j k}\)
  and
  \(\zeta^{d_j (d-d_j)} \sum_{k=1}^n \omega_{k j} u_{k i}^* =
  \sum_{k=1}^n \zeta^{d_j d_k} \omega_{k j} u_{k i}^*\) are multiplied
  by \(\zeta^{-s(d_i+d_j-s)- s d_k} = \zeta^{-s(d_j+d-s)}\) and
  \(\zeta^{-s(d_j + d_k -s)- s d_j} = \zeta^{-s(d_j+d-s)}\),
  respectively; here we use that \(d_j + d_k = d\) if
  \(\omega_{j k}\neq0\).  Since the two factors are the same and do not
  involve the summation index~\(k\) after simplification, the
  relation~\eqref{eq:xi_invariant} holds for \((\pi,\omega)\) if and
  only if it holds for \((\pi',\omega')\).  The same is obvious for the
  other two relations \eqref{eq:u_isometry}
  and~\eqref{eq:u_coisometry}.  Hence the generators~\(u_{ij}\) of
  \(A_o(V,\pi,\omega)\) and \(A_o(V,\pi',\omega')\) are subjected to the
  same relations.  The formulas for the \(\T\)\nb-action and the
  comultiplication in \eqref{eq:beta_on_u}
  and~\eqref{eq:comultiplication_u} also remain the same because
  \(d_i' - d_j' = d_i - d_j\) for all \(1 \le i,j \le n\).  So the
  canonical map \(A_o(V,\pi,\omega)\to A_o(V,\pi',\omega')\) is a
  \(\T\)\nb-equivariant Hopf \Star{}isomorphism.
\end{proof}

The formula for~\(\omega'\) in Lemma~\ref{lem:shift_degrees} may be
understood using the universal property in
Theorem~\ref{the:universal_Ao}.  Let \(\G=(B,\beta,\Comult)\) be
another braided quantum group.  The \(\T\)\nb-equivariant Hopf
\Star{}isomorphism in Lemma~\ref{lem:shift_degrees} means that
\(U\in \Bound(V) \otimes B\) is a representation of~\(\G\) on
\((V,\pi)\) for which \(\omega\in V\otimes V\) is \(d\)\nb-homogeneous
and \(\G\)\nb-invariant if and only if \(U\in \Bound(V) \otimes B\)
is a representation of~\(\G\) on \((V,\pi')\) for which
\(\omega\in V\otimes V\) is \(d'\)\nb-homogeneous and
\(\G\)\nb-invariant.  A \(d\)\nb-homogeneous invariant vector
corresponds to an intertwining operator
\(\omega\colon \C_d \to V\otimes V\).  And such an intertwining
operator is equivalent to an intertwining operator
\[
  \C_{d-2 s} \cong \C_{-s} \otimes \C_d \otimes \C_{-s}
  \xrightarrow{1\otimes\omega\otimes1}
  \C_{-s} \otimes V\otimes V \otimes \C_{-s}
  \xrightarrow[\cong]{\Braiding{V}{\C_{-s}}_{34}}
  (\C_{-s} \otimes V)\otimes (\C_{-s} \otimes V).
\]
Now \(\C_{-s} \otimes V\) is~\(V\) with the shifted
representation~\(\pi'\) and the same~\(U\), exactly as needed.  The
braiding operator \(\Braiding{V}{\C_{-s}}\) maps the basis vector
\(e_j \otimes 1\) of \(V\otimes \C_{-s}\) to
\(\zeta^{-s d_j}\cdot 1 \otimes e_j\).  This is the origin of the
factors~\(\zeta^{-s d_j}\) in the definition of~\(\omega'\).

The isomorphism in Lemma~\ref{lem:shift_degrees} allows to reduce
the case of even~\(d\) to \(d=0\), up to a canonical
\(\T\)\nb-equivariant Hopf \Star{}isomorphism.  We shall use this
simplification later to describe the representation category of
\(A_o(V,\pi,\omega)\).

The case \(d=1\) is particularly important because it occurs for
\(\qSU\).  We can also reduce this case to \(d=0\) if we replace the
circle~\(\T\) by a double cover.

Namely, we may compose the representation~\(\pi\) with the
endomorphism \(z\mapsto z^2\) of~\(\T\).  We may also use any other
endomorphism, but only this two-fold covering map will be important
later.  Replacing~\(\pi\) by \(\pi''_z = \pi_{z^2}\) replaces
\(d_i\) by~\(2 d_i\) for \(1\le i \le n\).  The \(d\)\nb-homogeneous
vector~\(\omega\) for~\(\pi\) is \(d''\)\nb-homogeneous for~\(\pi''\)
with \(d'' \defeq 2 d\).  In order to keep the braiding unitaries
the same, we replace~\(\zeta\) by a fourth root
\(\zeta'' = \sqrt[4]{\zeta}\).  We leave the vector \(\omega'' = \omega\)
unchanged.  Then the relations
\eqref{eq:u_isometry}--\eqref{eq:xi_invariant} for \((V,\pi,\omega)\)
and \((V,\pi'',\omega)\) are the same.
And~\eqref{eq:comultiplication_u} is also the same.  But the
action~\(\beta''\) on \(A_o(V,\pi'',\omega)\) defined
in~\eqref{eq:beta_on_u} becomes \(\beta''_z = \beta_{z^2}\).  Thus
there is a Hopf \Star{}isomorphism
\(\varphi\colon A_o(V,\pi,\omega) \congto A_o(V,\pi'',\omega)\) that maps
the generator~\(u_{ij}\) of \(A_o(V,\pi,\omega)\) to
\(u_{ij} \in A_o(V,\pi'',\omega)\).  But~\(\varphi\) is not
\(\T\)\nb-equivariant.  Instead, it satisfies
\(\beta''_z\circ \varphi = \varphi \circ \beta_{z^2}\) for all
\(z\in\T^2\).  Thus the braided quantum groups \(A_o(V,\pi,\omega)\)
and \(A_o(V,\pi'',\omega)\) are different.  They have different
representations.  Namely, representations of \(A_o(V,\pi,\omega)\) are
equivalent to those representations of \(A_o(V,\pi'',\omega)\) where the
underlying representation of~\(\T\) factors through the map
\(\T\to\T\), \(z\mapsto z^2\).

Since \(d'' = 2d\) is even, Lemma~\ref{lem:shift_degrees} for
\(s=d\) identifies \(A_o(V,\pi'',\omega)\) with \(A_o(V,\pi',\omega')\)
for \(\pi'_z = z^{-d} \pi''_z = z^{-d} \pi_{z^2}\),
\(\zeta' = \zeta'' = \sqrt[4]{\zeta}\), and
\(\omega'_{i j} = (\zeta')^{-d d_j''} \omega_{i j} = \zeta^{-d d_j/2}
\omega_{i j}\).  This gives a braided free orthogonal quantum group
with \(d'=0\).

\section{The bosonisation}
\label{sec:semidirect_product}

The bosonisation of a braided quantum group is an ordinary quantum
group together with an idempotent quantum group homomorphism
(``projection'').  Conversely, any quantum group with projection is
the bosonisation of a braided quantum group
(see~\cite{Meyer-Roy-Woronowicz:Qgrp_proj}).  In the compact case,
the bosonisation is constructed in
\cite{Meyer-Roy-Woronowicz:Twisted_tensor_2}*{Corollary~6.5}, where
it is called ``semidirect product''.  When we specialise this
construction to the braided quantum group \(A_o(V,\pi,\omega)\), we
get the following:

\begin{theorem}
  \label{the:semidirect_product}
  Let \((C,\Comult[C])\) be the bosonisation 
  associated to the braided quantum group~\(A_o(V,\pi,\omega)\).
  Then~\(C\) is isomorphic to the universal unital
  \(\Cst\)\nb-algebra generated by elements \(z\) and \(u_{i j}\)
  for \(1 \le i,j \le n\) subject to the relations
  \(z z^* = 1 = z^* z\),
  \eqref{eq:u_isometry}--\eqref{eq:xi_invariant} and the commutation
  relations \(z u_{i j} = \zeta^{d_i-d_j} u_{i j} z\) for all
  \(1 \le i,j \le n\).  The comultiplication~\(\Comult[C]\) is given
  by
  \[
    \Comult[C](z) = z\otimes z,\qquad
    \Comult[C](u_{i k})
    = \sum_{l=1}^n u_{i l} \otimes z^{d_l-d_i} u_{l k}.
  \]
\end{theorem}

\begin{proof}
  In the notation of~\cite{Meyer-Roy-Woronowicz:Twisted_tensor_2},
  we have the following data.  First, \(A= \Cont(\T)\) with the
  usual comultiplication from the Abelian group structure on~\(\T\)
  and \(B= A_o(V,\pi,\omega)\), equipped with the continuous coaction
  \(\beta\colon B\to B\otimes A\) defined by~\eqref{eq:beta_on_u}.
  The braided tensor product
  in~\cite{Meyer-Roy-Woronowicz:Twisted_tensor_2} is defined for
  Yetter--Drinfeld algebras over~\(A\), which means here algebras
  with an action of the group \(\T\times\hat{\T} = \T\times\Z\).
  The relevant \(\Z\)\nb-action on~\(B\) is defined by
  composing~\(\beta\) with the homomorphism \(\Z \to \T\),
  \(n\mapsto \zeta^n\).  This ensures that the braiding unitaries
  used in~\cite{Meyer-Roy-Woronowicz:Twisted_tensor_2} to define the
  braided tensor product are the same ones as in
  \eqref{eq:Braiding_L_V}, \eqref{eq:Braiding_L_L}
  and~\eqref{eq:Braiding_LL_V}.

  Represent \(\Cont(\T)\) faithfully on \(L^2(\T)\) by pointwise
  multiplication, and represent \(A_o(V,\pi,\omega)\) faithfully and
  \(\T\)\nb-equivariantly on~\(\Hils[L]\) as in the proof of
  Proposition~\ref{pro:universal_property}.  Defined the braiding
  operators
  \[
    \Braiding{\Hils[L]}{L^2(\T)}(\lambda_m \otimes z^k) \defeq
    \zeta^{k\cdot l_m} \cdot z^k \otimes \lambda_m,\qquad
    \Dualbraiding{L^2(\T)}{\Hils[L]} (z^k \otimes \lambda_m) \defeq
    \zeta^{-k\cdot l_m} \cdot \lambda_m \otimes z^k,
  \]
  using the \(\T\)\nb-homogeneous basis~\((\lambda_m)\)
  of~\(\Hils[L]\).  The underlying \(\Cst\)\nb-algebra of the
  bosonisation is the braided tensor product
  \begin{multline*}
    C \defeq \Cont(\T) \boxtimes A_o(V,\pi,\omega)
    \\= (\Cont(\T) \otimes 1) \cdot \Braiding{\Hils[L]}{L^2(\T)}
    \cdot (A_o(V,\pi,\omega) \otimes 1)
    \cdot \Dualbraiding{L^2(\T)}{\Hils[L]}
    \subseteq \Bound(L^2(\T) \otimes \Hils[L]).
  \end{multline*}
  Let \(z \in \Cont(\T) \subseteq \Bound(L^2\T)\) denote the unitary
  operator of pointwise multiplication with the identity function.
  This generates \(\Cont(\T)\).  Since \(A_o(V,\pi,\omega)\) is
  generated by the matrix coefficients~\(u_{ij}\) of~\(u\), the
  \(\Cst\)\nb-algebra~\(C\) is generated by the operators
  \(j_1(z) \defeq z\otimes 1\) and
  \(j_2(u_{ij}) \defeq \Braiding{\Hils[L]}{L^2(\T)} \cdot (u_{ij}
  \otimes 1) \cdot \Dualbraiding{L^2(\T)}{\Hils[L]}\), which acts on
  \(L^2(\T) \otimes \Hils[L]\) by
  \begin{align}
    \label{eq:j2_u}
    j_2(u_{ij})(z^k \otimes \lambda_m)
    &= \bigl(\Braiding{\Hils[L]}{L^2(\T)} \cdot (u_{ij} \otimes 1) \cdot
    \Dualbraiding{L^2(\T)}{\Hils[L]}\bigr) (z^k \otimes \lambda_m)
    \\&= \zeta^{-k l_m} \cdot
    \Braiding{\Hils[L]}{L^2(\T)}(u_{ij} \lambda_m \otimes z^k) \notag
    \\&= \zeta^{k(d_j-d_i)} z^k \otimes u_{ij} \lambda_m. \notag
  \end{align}
  This implies the commutation relation
  \[
    j_1(z) j_2(u_{ij}) j_1(z^*)
    = \zeta^{-(d_j-d_i)}\cdot j_2(u_{ij})
    = j_2(\beta_{\zeta^{-1}}(u_{ij})).
  \]
  Thus the unitaries \(j_1(z^n)\) for \(n\in\Z\) and the
  representation~\(j_2\) of \(A_o(V,\pi,\omega)\) form a covariant
  representation for the \(\Z\)\nb-action on \(A_o(V,\pi,\omega)\)
  generated by the automorphism~\(\beta_{\zeta^{-1}}\).  This
  covariant representation is unitarily equivalent to the regular
  representation that defines the reduced crossed product.  So
  \[
    C = \Cont(\T) \boxtimes A_o(V,\pi,\omega)
    \cong A_o(V,\pi,\omega) \rtimes \Z
  \]
  for the \(\Z\)\nb-action on~\(A_o(V,\pi,\omega)\) generated
  by~\(\beta_{\zeta^{-1}}\).  This crossed product
  \(\Cst\)\nb-algebra is also the universal unital
  \(\Cst\)\nb-algebra generated by elements
  \((u_{i j})_{1\le i,j\le n}\) and~\(z\) subject to the relations
  \eqref{eq:u_isometry}--\eqref{eq:xi_invariant}, the relation
  \(z z^* = 1 = z^* z\) saying that~\(z\) is unitary, and the
  commutation relation
  \[
    z u_{i j} = \zeta^{d_i-d_j} u_{i j} z.
  \]
  The comultiplication \(\Comult[C]\colon C \to C\otimes C\) is
  defined in~\cite{Meyer-Roy-Woronowicz:Twisted_tensor_2} as the
  composite of
  \[
    \Id_A \boxtimes\Comult\colon A \boxtimes B
    \to A \boxtimes B \boxtimes B
  \]
  and the unique \Star{}homomorphism
  \(\psi\colon A \boxtimes B \boxtimes B \to (A \boxtimes B)\otimes
  (A \boxtimes B)\) with
  \[
    \psi(j_1(a)) = (j_1 \otimes j_1)\Comult[A](a),\qquad
    \psi(j_2(b)) = (j_2 \otimes j_1)\beta(b),\qquad
    \psi(j_3(b)) = 1 \otimes j_2(b)
  \]
  for all \(a\in A\), \(b\in B\).  Here~\(j_l\) denotes the
  embedding of the \(l\)th factor in a braided tensor product, and
  \(\beta\colon B \to B\otimes A\) is the \(\T\)\nb-action on~\(B\).
  Using \(\Comult[A](z) = z\otimes z\)
  and~\eqref{eq:comultiplication_u}, we compute
  \begin{align*}
    \Comult[C](z)
    &= z\otimes z,\\
    \Comult[C](u_{i k})
    &= \psi\left(\sum_{l=1}^n j_2(u_{i l}) j_3(u_{l k})\right)
      = \sum_{l=1}^n (j_2\otimes j_1)(u_{i l} \otimes z^{d_l-d_i})
      \cdot (1\otimes j_2(u_{l k}))
    \\&= \sum_{l=1}^n j_2(u_{i l}) \otimes j_1(z^{d_l-d_i}) j_2(u_{l k}).
  \end{align*}
  This gives the formulas for~\(\Comult[C]\) in the theorem when we
  drop the inclusion maps \(j_1\) and~\(j_2\) from our notation as
  above.
\end{proof}

A sanity check for the formulas in
Theorem~\ref{the:semidirect_product} is that~\(\Comult[C]\) is
coassociative:
\begin{align*}
  (\Id \otimes \Comult[C]) \Comult[C](z)
  = (\Comult[C] \otimes \Id) \Comult[C](z)
  &= z\otimes z \otimes z,\\
  (\Id \otimes \Comult[C]) \Comult[C](u_{i k})
  = (\Comult[C] \otimes \Id) \Comult[C](u_{i k})
  &= \sum_{l,m=1}^n
    u_{i l} \otimes z^{d_l - d_i} u_{l m} \otimes z^{d_m - d_i} u_{m k}.
\end{align*}
It can also be checked directly that~\(\Comult[C]\) is well defined,
that is, the elements \(\Comult[C](z)\) and \(\Comult[C](u_{i k})\)
of \(C\otimes C\) verify the relations in the universal property
of~\(C\).

Representations of the braided quantum group \(A_o(V,\pi,\omega)\) are
equivalent to representations of the bosonisation by
\cite{Meyer-Roy:Braided_mu}*{Theorem~3.4}.  Here a representation of
\(A_o(V,\pi,\omega)\) on a Hilbert space~\(\Hils\) is a pair
\((\Corep{S},\Corep{U})\), where
\(\Corep{S} \in \U(\Comp(\Hils) \otimes \Cont(\T)) \subseteq
\U(\Hils \otimes L^2(\T))\) is a representation of~\(\T\) and
\(\Corep{U} \in \U(\Comp(\Hils) \otimes A_o(V,\pi,\omega)) \subseteq
\U(\Hils \otimes \Hils[L])\) is a representation of
\(A_o(V,\pi,\omega)\).  By \cite{Meyer-Roy:Braided_mu}*{Theorem~3.4},
the corresponding representation of
\(C\subseteq \Bound(L^2(\T) \otimes \Hils[L])\) is
\[
  \Corep{S}_{12} (\Corep{U} \tenscorep 1_{L^2(\T)})
  = \Corep{S}_{12} \cdot \Braiding{\Hils[L]}{L^2(\T)}_{23} \cdot
  \Corep{U}_{12} \cdot \Dualbraiding{L^2(\T)}{\Hils[L]}_{23}
  \in \U(\Hils \otimes L^2(\T) \otimes \Hils[L]).
\]
In particular, the fundamental representation of~\(A_o(V,\pi,\omega)\)
on~\(V\) is given by the pair \((\Corep{S}^V,\Corep{U}^V)\) with
\[
  \Corep{S}^V(e_i \otimes z^k) = e_i \otimes z^{d_i+k},\qquad
  \Corep{U}^V(e_i \otimes \lambda_m)
  = \sum_{k=1}^n e_k \otimes u_{k i} \lambda_m.
\]
The corresponding representation~\(t\) of~\(C\) is given by the
formula
\begin{multline*}
  t(e_i \otimes z^k \otimes \lambda_m)
  = \Corep{S}_{12} \Braiding{\Hils[L]}{L^2(\T)}_{23}
  \Corep{U}_{12} \Dualbraiding{L^2(\T)}{\Hils[L]}_{23}
  (e_i \otimes z^k \otimes \lambda_m)
  \\= \zeta^{-k l_m} \Corep{S}_{12} \Braiding{\Hils[L]}{L^2(\T)}_{23}
  \Corep{U}_{12} (e_i \otimes \lambda_m \otimes z^k)
  = \sum_{h=1}^n \zeta^{-k l_m} \Corep{S}_{12}
  \Braiding{\Hils[L]}{L^2(\T)}_{23} (e_h \otimes u_{h i}\lambda_m
  \otimes z^k)
  \\= \sum_{h=1}^n \zeta^{-k l_m+ k (l_m + d_i - d_h)} \Corep{S}_{12}
  (e_h \otimes z^k \otimes u_{h i}\lambda_m)
  = \sum_{h=1}^n \zeta^{k (d_i - d_h)}
  (e_h \otimes z^{k+d_h} \otimes u_{h i}\lambda_m).
\end{multline*}
This and~\eqref{eq:j2_u} show that \(t \in \Mat_n(C)\) has matrix
entries \(t_{h i} = j_1(z^{d_h}) j_2(u_{h i})\) or briefly
\[
  t_{h i} = z^{d_h} u_{h i}.
\]
A sanity check is that~\(t\) is a representation of~\(C\).
Indeed, we compute that \(\Comult[C](t) = t_{12} t_{13}\) is the
matrix over \(C\otimes C\) with \(i,k\)th entry
\[
  \sum_{m=1}^n z^{d_i} u_{i m} \otimes z^{d_m} u_{m k}.
\]

Since \(t_{i j} = z^{d_i} u_{i j}\) and
\(z^{-d_i} t_{i j} = u_{i j}\), we may also describe~\(C\) as the
universal unital \(\Cst\)\nb-algebra generated by \(z\)
and~\(t_{i j}\) for \(1 \le i, j \le n\).  These are subject to the
following relations.  First, \(z\) and the matrix \(t = (t_{i j})\)
are unitary.  Secondly, \(z t_{i j} = \zeta^{d_i - d_j} t_{i j} z\)
for \(1 \le i, j \le n\).  And third, \eqref{eq:xi_invariant}
becomes
\[
  \zeta^{d_j d_i} \sum_{k=1}^n \omega_{i k} z^{-d_j} t_{j k}
  = \zeta^{d_j (d-d_j)} \sum_{k=1}^n \omega_{k j} (z^{-d_k} t_{k i})^*
  = \zeta^{(d-d_i) (d-d_j)} \sum_{k=1}^n \omega_{k j} z^{d-d_j} t_{k i}^*
\]
because \(\omega_{k j}=0\) unless \(d_k + d_j = d\) and
\(t_{k i}^* z^{d-d_j} = \zeta^{(d-d_j)\cdot (d_k - d_i)} z^{d-d_j}
t_{k i}^*\).  Cancelling~\(z^{-d_j}\) and multiplying with
\(\zeta^{(d-d_j) d_i}\), this becomes
\begin{equation}
  \label{eq:xi_invariant_for_t}
  \sum_{k=1}^n t_{j k} (\zeta^{d d_i} \omega_{i k})
  = z^d\sum_{k=1}^n (\zeta^{d d_k} \omega_{k j}) t_{k i}^*.
\end{equation}

Since \((C,\Comult[C])\) is a bosonisation of a
braided quantum group, it comes with quantum group morphisms
\((C,\Comult[C]) \leftrightarrow (A,\Comult[A])\) whose composite
makes \((C,\Comult[C])\) a \(\Cst\)\nb-quantum group with projection
(see~\cite{Meyer-Roy-Woronowicz:Qgrp_proj}).  We also describe these
quantum group morphisms explicitly:

\begin{proposition}
  \label{pro:projection_on_C}
  The quantum group morphisms
  \((C,\Comult[C]) \leftrightarrow (A,\Comult[A])\) are given by the
  Hopf \Star{}homomorphisms \(\iota\colon A\to C\) and
  \(\pi\colon C \to A\) defined by \(\iota(z) \defeq z\),
  \(\pi(z)\defeq z\) and \(\pi(u_{i j}) \defeq \delta_{i j}\).
\end{proposition}

\begin{proof}
  We interpret a representation of \((C,\Comult[C])\) as a
  braided representation \((\Corep{S},\Corep{U})\) of
  \(A_o(V,\pi,\omega)\).  Then the functor induced by~\(\iota\) maps a
  representation~\(\Corep{S}\) of \(A=\Cont(\T)\) (that is, a
  representation of~\(\T\)) to the representation
  \((\Corep{S},1)\) of \(A_o(V,\pi,\omega)\) on the same Hilbert space.
  And the functor induced by~\(\pi\) maps a
  representation~\((\Corep{S},\Corep{U})\) of \(A_o(V,\pi,\omega)\)
  to the underlying representation~\(\Corep{S}\) of \(\Cont(\T)\)
  on the same Hilbert space.  By
  \cite{Meyer-Roy:Braided_mu}*{Proposition~2.15}, a functor between
  the representation categories of two \(\Cst\)\nb-quantum groups
  (defined by manageable multiplicative unitaries) is equivalent to
  a quantum group morphism as defined
  in~\cite{Meyer-Roy-Woronowicz:Homomorphisms}.  And it is checked
  in \cite{Meyer-Roy:Braided_mu}*{Proposition~3.5} that the quantum
  group morphisms that correspond to the functors above are the ones
  defined in the study of \(\Cst\)\nb-quantum groups with projection
  in~\cite{Meyer-Roy-Woronowicz:Qgrp_proj}.

  It is easy to check that there are unique \Star{}homomorphisms
  \(\iota\) and~\(\pi\) given by the formulas in the statement of
  the proposition and that these are Hopf \Star{}homomorphisms.
  These Hopf \Star{}homomorphisms induce functors between the
  representation categories of \((C,\Comult[C])\) and
  \((A,\Comult[A])\) that do not change the underlying Hilbert
  space.  A computation shows that the functors on the
  representation categories corresponding to \(\iota\) and~\(\pi\)
  map \(\Corep{S}\) to \((\Corep{S},1)\) and
  \((\Corep{S},\Corep{U})\) to \(\Corep{S}\), respectively.
\end{proof}

\section{The representation category}
\label{sec:rep_category}

We are going to prove Theorem~\ref{the:irrep_tensor}.  In the
previous section, we have described the bosonisation 
\((C,\Comult[C])\) and translated representations of the
braided quantum group \(A_o(V,\pi,\omega)\) to representations of
\((C,\Comult[C])\).  In particular, \(z \in C\) and
\(t\in \Mat_n(C)\) are representations of~\((C,\Comult[C])\).  The
representation~\(z\) corresponds to the trivial representation
of~\(A_o(V,\pi,\omega)\) on the \(1\)\nb-dimensional vector space on
which~\(\T\) acts by the identity character \(\T \to \U(1)\),
\(z\mapsto z\).  The representation~\(t\) corresponds to the
fundamental representation of~\(A_o(V,\pi,\omega)\) on~\(V\).

\begin{corollary}
  \label{cor:CMG}
  The \(\Cst\)\nb-algebra~\(C\) is generated by the
  coefficients of the direct sum representation \(z\oplus t\).  And
  \((C,z\oplus t)\) is a compact matrix quantum group.
\end{corollary}

\begin{proof}
  The coefficients of \(z\oplus t\) are~\(z\) and the
  coefficients~\(t_{ij}\) of~\(t\).  We have seen in the last
  section that these elements generate~\(C\) as a
  \(\Cst\)\nb-algebra.  The representation~\(z\) is a character, so
  that \(\conj{z} = z^* = z^{-1}\).
  Equation~\eqref{eq:xi_invariant_for_t} is equivalent to an
  equality of matrices \(t\cdot F = z^d \cdot F\cdot \conj{t}\) for
  \(t = (t_{i j})_{1 \le i, j \le n}\),
  \(\conj{t} = (t^*_{i j})_{1 \le i,j\le n}\) and \(F = (F_{i j})\)
  with
  \begin{equation}
    \label{eq:F_from_xi}
    F_{i j} \defeq \zeta^{d d_j} \omega_{j i}.
  \end{equation}
  The elements~\(t_{ij}\) are linearly independent for all~
  \(1\leq i,j\leq n\).  Equation~\ref{eq:Xi_condition_2} implies
  that~\(F\) is invertible. 
  Since \(t\) and~\(z\) are unitary, it follows that
  \(\conj{t} = F^{-1} z^{-d} t\cdot F\) is invertible.  Thus
  \(\conj{z} \oplus \conj{t}\) is invertible as well.  And this
  means that~\((C,t\oplus z)\) is a compact matrix quantum group as
  defined in~\cite{Woronowicz:Compact_pseudogroups}.
\end{proof}

The representation theory of compact matrix quantum groups was
studied by Woronowicz in~\cite{Woronowicz:Compact_pseudogroups}.  It
follows that any irreducible representation of~\((C,\Comult[C])\) is
contained in a tensor product of several copies of \(z\), \(t\),
\(\conj{z}\) and~\(\conj{t}\).  The equation
\(t\cdot F = z^d \cdot F\cdot \conj{t} = F\cdot z^d \cdot \conj{t}\)
says that~\(F\) is an intertwiner from \(z^d \cdot \conj{t}\)
to~\(t\).  Equivalently, \(F\) is an intertwiner from \(\conj{t}\)
to~\(z^{-d} t\).  Hence the tensor factor~\(\conj{t}\) is redundant.
The commutation relations \(z t_{ij} = \zeta^{d_i-d_j} t_{ij} z\)
say that~\(\pi_\zeta\), the diagonal matrix with
entries~\(\zeta^{d_i}\), is an intertwiner from \(t\tenscorep z\) to
\(z\tenscorep t\).  Hence
\(t \tenscorep z^k \cong z^k \tenscorep t\) for all \(k\in \Z\).  As
a result, any irreducible representation of~\((C,\Comult[C])\) is a
direct summand in \(z^k \tenscorep t^{\tenscorep l}\) for some
\(k\in\Z\), \(l\in\N\).

Now we are going to relate the quantum group \((C,\Comult[C])\) to
the usual free
orthogonal quantum groups~\(A_o(F)\) of Wang and van
Daele~\cites{Daele-Wang:Universal,Wang:Free_products_of_CQG}.
Recall that~\(A_o(F)\) for an invertible matrix
\(F\in\mathrm{Gl}_n(\C)\) is the universal \(\Cst\)\nb-algebra with
generators~\(x_{ij}\) subject to the relations
\begin{align}
  \label{eq:x_isometry}
  \sum_{k=1}^n x^*_{k i} x_{k j} &= \delta_{i,j},\\
  \label{eq:x_coisometry}
  \sum_{k=1}^n x_{i k} x^*_{j k} &= \delta_{i,j},\\
  \label{eq:F_intertwines}
  \sum_{k=1}^n x_{i k} F_{k j} &= \sum_{k=1}^n F_{i k} x^*_{k j}
  \qquad\text{for }1\le i,j\le n.
\end{align}
Equivalently, \(x=(x_{i j})_{1\le i,j\le n}\in \Mat_n(A_o(F))\) is
unitary and satisfies \(x\cdot F = F\cdot \conj{x}\).  The
comultiplication on \(A_o(F)\) is the unique \Star{}homomorphism
\(\Comult[A_o(F)]\) with
\[
  \Comult[A_o(F)](x_{i j}) = \sum_{k=1}^n x_{i k} \otimes x_{k j}.
\]

Now we assume for some time that \(d=0\).  Then~\eqref{eq:F_from_xi}
simplifies to \(F_{i j} = \omega_{j i}\) or \(F = \Omega^\top\).
Equation~\eqref{eq:Xi_condition_2} says that \(F \conj{F} = c\cdot
1\) because \(d=0\).  So \(A_o(F)\) is defined and the results of
Banica~\cite{Banica:Rep_On} apply.

By assumption, \(F_{i j} = \omega_{j i}\) vanishes unless
\(d_i + d_j = 0\).  Therefore, multiplying each~\(x_{i j}\)
by~\(\zeta^{d_i - d_j}\) preserves the relations
\eqref{eq:x_isometry}--\eqref{eq:F_intertwines} that define
\(A_o(F)\).  Then there is an automorphism~\(\alpha\) of the
\(\Cst\)\nb-algebra \(A_o(F)\) with
\(\alpha(x_{i j}) \defeq \zeta^{d_i-d_j} x_{i j}\) for
\(1 \le i,j \le n\).  This is even a Hopf \Star{}homomorphism, that
is,
\(\Comult[A_o(F)]\circ \alpha = (\alpha \otimes \alpha) \circ
\Comult[A_o(F)]\).  There is also a unique unital Hopf
\Star{}homomorphism \(\varphi\colon A_o(F) \to C\) mapping
\(x_{i j} \mapsto t_{i j}\) for \(1 \le i,j \le n\).  This follows
from the universal property of~\(A_o(F)\) because \(t\in \Mat_n(C)\)
is a representation of~\((C,\Comult[C])\) and~\(F\) is an
intertwiner from~\(\conj{t}\) to~\(t\).  The following proposition
implies that~\(\varphi\) is injective:

\begin{proposition}
  \label{pro:relate_to_AoF}
  Let \(d=0\).  The Hopf \Star{}homomorphism
  \(\varphi\colon A_o(F) \to C\) extends to an isomorphism
  \(A_o(F) \rtimes_\alpha \Z \cong C\).
\end{proposition}

\begin{proof}
  The crossed product \(A_o(F)\rtimes \Z\) is generated by the
  generators~\(x_{i j}\) of~\(A_o(F)\) and an extra unitary~\(z\),
  subject to the relations
  \eqref{eq:x_isometry}--\eqref{eq:F_intertwines}
  defining~\(A_o(F)\), the unitarity conditions
  \(z^* z = z z^* = 1\), and the relation
  \(z x_{i j} z^* = \alpha(x_{i j}) = \zeta^{d_i-d_j} x_{i j}\) for
  \(1 \le i,j\le n\).  When we replace~\(x_{i j}\) by~\(t_{i j}\),
  this gives a presentation of~\(C\) by the computations in
  Section~\ref{sec:semidirect_product}, using \(d=0\).
\end{proof}

Using the Hopf \Star{}homomorphism~\(\varphi\), a representation
\(\varrho \in \Mat_n(A_o(F))\) induces a representation in
\(\varphi_*(\varrho)\in \Mat_n(C)\).  A linear map
\(\psi \colon \C^n \to \C^n\) is an intertwiner for~\(\varrho\) if
and only if it is an intertwiner for \(\varphi_*(\varrho)\)
because~\(\varphi\) is injective.  So~\(\varphi_*\) is a fully
faithful functor from the representation category of~\(A_o(F)\) to
that of~\(C\).  It is also a strict tensor functor because
\(\varphi_*(\varrho_1 \tenscorep \varrho_2) = \varphi_*(\varrho_1)
\tenscorep \varphi_*(\varrho_2)\) and
\(\varphi_*(\varepsilon) = \varepsilon\) for the trivial
representation \(\varepsilon = 1\in A_o(F)\).  The representation
category of~\(A_o(F)\) is described by Banica~\cite{Banica:Rep_On}:
he finds that~\(A_o(F)\) has irreducible representations~\(r_k\) for
\(k\in\N\) such that \(\conj{r_k} \cong r_k\) for all \(k\in\N\) and
\[
  r_k \tenscorep r_s \cong r_{\abs{k-s}} \oplus r_{\abs{k-s}+2}
  \oplus \dotsb \oplus r_{k+s-2} \oplus r_{k+s}.
\]
These are the same fusion rules as for the group \(\mathrm{SU(2)}\).
Using the fully faithful strict tensor functor~\(\varphi_*\), we get
irreducible representations \(\varphi_*(r_k)\) of~\((C,\Comult[C])\)
with the same fusion rules.

\begin{lemma}
  \label{lem:orthogonal_pieces_in_rep_cat}
  Let \(d=0\).
  Let \(\varrho_1,\varrho_2\) be representations of \(A_o(F)\) and
  let \(a,b\in\Z\).  If \(a\neq b\), then there are no intertwiners
  between the representations
  \(z^a \tenscorep \varphi_*(\varrho_1)\) and
  \(z^b \tenscorep \varphi_*(\varrho_2)\).  If \(a=b\), then the two
  maps
  \[
    \Mor(\varrho_1,\varrho_2) \rightrightarrows
    \Mor\bigl(z^a \tenscorep \varphi_*(\varrho_1),
    z^a \tenscorep \varphi_*(\varrho_2)\bigr),\quad
    \Psi \mapsto
    \left\{\begin{array}{l}
      1 \otimes \Psi,\\
      \Dualbraiding{\varphi_*(\varrho_2)}{z^a}(\Psi \otimes 1)
      \Braiding{z^a}{\varphi_*(\varrho_1)},
    \end{array}\right.
  \]
  are equal and invertible.
\end{lemma}

\begin{proof}
  We have already seen that the braiding unitary~\(\pi_\zeta\) is a
  unitary intertwiner between
  \(z \tenscorep t= z\tenscorep \varphi_*(x)\) and
  \(t \tenscorep z = \varphi_*(x) \tenscorep z\).  Hence the
  appropriate braiding unitary is a unitary intertwiner
  \(z \tenscorep \varphi_*(x^{\tenscorep n}) \cong
  \varphi_*(x^{\tenscorep n}) \tenscorep z\).  Then the braiding
  unitary is a unitary intertwiner
  \(z \tenscorep \varphi_*(\varrho) \cong \varphi_*(\varrho)
  \tenscorep z\) for any representation~\(\varrho\) of~\(A_o(F)\)
  because~\(\varrho\) is a direct sum of direct summands in the
  tensor powers~\(x^{\tenscorep n}\) for \(n\in\N\).  This also
  implies
  \(z^a \tenscorep \varphi_*(\varrho) \cong \varphi_*(\varrho)
  \tenscorep z^a\) for all \(a\in\Z\) and all
  representations~\(\varrho\) of~\(A_o(F)\).

  The matrix coefficients of \(z^a \tenscorep \varphi_*(\varrho_1)\)
  belong to \(A_o(F)\cdot z^a \subseteq A_o(F) \rtimes \Z \cong C\).
  If \(a\neq b\), then the subspaces \(A_o(F)\cdot z^a\) and
  \(A_o(F)\cdot z^b\) of \(A_o(F) \rtimes \Z\) are linearly
  independent.  This implies that there cannot be any non-zero
  intertwiners between \(z^a \tenscorep \varphi_*(\varrho_1)\) and
  \(z^b \tenscorep \varphi_*(\varrho_2)\).

  % A non-zero intertwiner from
  % \(z^a \tenscorep \varphi_*(\varrho_1)\) to
  % \(z^b \tenscorep \varphi_*(\varrho_2)\) induces a non-zero
  % intertwiner from the trivial representation to
  % \[
  %   z^a \tenscorep \varphi_*(\varrho_1)\tenscorep
  %   \conj{(z^b \tenscorep \varphi_*(\varrho_2))}
  %   \cong z^a \tenscorep \varphi_*(\varrho_1\tenscorep
  %   \conj{\varrho_2}) \tenscorep z^{-b}
  %   \cong z^{a-b} \tenscorep \varphi_*(\varrho_1\tenscorep
  %   \conj{\varrho_2}).
  % \]
  % The trivial representation is irreducible, and
  % \(z^{a-b} \tenscorep \varphi_*(\varrho_1\tenscorep
  % \conj{\varrho_2})\) is a direct sum of the representations
  % \(z^{a-b} \tenscorep \varphi_*(r_k)\) for the irreducible
  % representations \(r_k\) of~\(A_o(F)\).  A non-zero intertwiner as
  % above exists if and only if \(z^{a-b} \tenscorep \varphi_*(r_k)\)
  % is the trivial representation for some \(k\in\N\).  Since~\(r_k\)
  % is only a character if \(k=0\) and since the characters~\(z^a\)
  % for \(a\in\Z\) are all different, this can only happen if \(a=b\).
  % This proves the claim in the case \(a \neq b\).

  Now assume \(a=b\).  We have already seen that there are the same
  intertwiners \(\varrho_1 \to \varrho_2\) and
  \(\varphi_*(\varrho_1) \to \varphi_*(\varrho_2)\).  Tensoring with
  the identity on~\(z^a\) on the left gives an intertwiner
  \(z^a \tenscorep \varphi_*(\varrho_1) \to z^a \tenscorep
  \varphi_*(\varrho_2)\).  This map between intertwiner spaces is
  invertible because tensoring with the identity on~\(z^{-a}\) gives
  an inverse map, using \(z^{-a} \tenscorep z^a = \varepsilon\), the
  trivial representation.  Similarly, tensoring on the right with
  the identity on~\(z^a\) gives an intertwiner
  \(\varphi_*(\varrho_1) \tenscorep z^a\to \varphi_*(\varrho_2)
  \tenscorep z^a\).  Conjugating the latter with the braiding
  unitaries
  \(z^a \tenscorep \varphi_*(\varrho_i) \cong \varphi_*(\varrho_i)
  \tenscorep z^a\) for \(i=1,2\) gives another intertwiner
  \(z^a \tenscorep \varphi_*(\varrho_1) \to z^a \tenscorep
  \varphi_*(\varrho_2)\).  We claim that these two intertwiners are
  equal.  Since any representation is contained in a direct sum of
  copies of~\(x^{\tenscorep m}\), it is enough to show this in case
  \(\varrho_i = x^{\tenscorep m_i}\) for \(i=1,2\) for some
  \(m_1,m_2\in\N\).  Banica shows that any intertwiner between the
  representations \(x^{\tenscorep m_i}\) is a \Star{}polynomial in
  the special intertwiners
  \(x^{\tenscorep m} \to x^{\tenscorep m+2}\) that add the invariant
  vector in two consecutive entries.  Therefore, it is enough to
  prove the claim for these special intertwiners.  In this case, the
  result follows because the invariant vector~\(\omega\) is
  \(\T\)\nb-invariant.
\end{proof}

The lemma implies, in particular, that the representations
\(z^a \tenscorep \varphi_*(r_k)\) of~\((C,\Comult[C])\) for
\(a\in\Z\), \(k\in\N\) are all irreducible and distinct.  And any
irreducible representation is one of them because it is
contained in \(z^a \tenscorep t^{\tenscorep l}\) for some
\(a\in\Z\), \(l\in\N\).  As a result, any irreducible representation
of~\((C,\Comult[C])\) is unitarily equivalent to
\(z^a \tenscorep r_k\) for a unique \(a\in\Z\), \(k\in\N\).

These irreducible representations satisfy the fusion rules
\begin{align*}
  \conj{z^k \tenscorep r_l}
  &\cong \conj{r_l} \tenscorep \conj{z^k}
  \cong r_l \tenscorep z^{-k}
  \cong z^{-k} \tenscorep r_l,\\
  (z^b \tenscorep r_a) \tenscorep (z^k \tenscorep r_m)
  &\cong z^{b+k} \tenscorep (r_{\abs{a-m}} \oplus r_{\abs{a-m}+2}
  \oplus \dotsb \oplus r_{a+m-2} \oplus r_{a+m}).
\end{align*}
These are the rules asserted in Theorem~\ref{the:irrep_tensor}.
Thus we have completed the proof of that theorem in the case \(d=0\).

Next, we describe the representation category \(\Rep(C,\Comult[C])\)
of~\((C,\Comult[C])\) for \(d=0\) as a monoidal category.  Let
\(\mathcal{R}_a\subseteq \Rep(C,\Comult[C])\) for \(a\in\Z\) be the
full subcategory of all representations of~\((C,\Comult[C])\) of the
form \(z^a \tenscorep \varphi_*(\varrho)\) for a
representation~\(\varrho\) of~\(A_o(F)\).
Lemma~\ref{lem:orthogonal_pieces_in_rep_cat} implies that there are
no non-zero intertwiners between representations in
\(\mathcal{R}_a\) and~\(\mathcal{R}_b\) for \(a\neq b\) and that the
intertwiners between two representations in~\(\mathcal{R}_a\) are
the same as for the corresponding representations of~\(A_o(F)\).
Thus we get an equivalence of categories
\(\Rep(C,\Comult[C]) \simeq \prod_{a\in\Z} \mathcal{R}_a\) and
isomorphisms of categories \(\mathcal{R}_a \cong \Rep(A_o(F))\) for
all \(a\in\Z\).  The isomorphism
\(\mathcal{R}_0 \cong \Rep(A_o(F))\) is~\(\varphi_*\), which is a
strict tensor functor and thus an isomorphism of monoidal
categories.  The tensor product of two representations in
\(\mathcal{R}_a\) and~\(\mathcal{R}_b\) belongs
to~\(\mathcal{R}_{a+b}\), and
\[
  (z^a \tenscorep \varphi_*(\varrho_1))
  \tenscorep (z^b \tenscorep \varphi_*(\varrho_2))
  \cong z^a \tenscorep z^b \tenscorep \varphi_*(\varrho_1)
  \tenscorep \varphi_*(\varrho_2)
  = z^{a+b} \tenscorep \varphi_*(\varrho_1 \tenscorep \varrho_2),
\]
where the intertwiner is the braiding unitary
\(\Braiding{\varphi_*(\varrho_1)}{\C_{z^b}}_{23}\).  The naturality
of the braiding unitary says that it commutes with \(S\otimes T\) if
\(S\in \Bound(\C_{z^b})\) and \(T\in \Bound(\varphi_*(\varrho_1))\)
are intertwiners or, more generally, \(\T\)\nb-equivariant.
Therefore, when we identify \(\mathcal{R}_a \cong \Rep(A_o(F))\) as
above for all \(a\in\Z\), then the tensor product functor on
\(\Rep(C,\Comult[C])\) restricted to a functor
\(\mathcal{R}_a \times \mathcal{R}_b \to \mathcal{R}_{a+b}\) becomes
the usual tensor product functor on \(\Rep(A_o(F))\).  The
associators
\((\varrho_1\tenscorep \varrho_2)\tenscorep \varrho_3 \cong
\varrho_1\tenscorep (\varrho_2\tenscorep \varrho_3)\) and the unit
transformations
\(\varepsilon \tenscorep \varrho \cong \varrho \cong \varrho
\tenscorep \varepsilon\) also restrict to the usual ones.  So the
representation category of~\((C,\Comult[C])\) is equivalent to the
representation category of \(A_o(F)\otimes \Cont(\T)\).  In other
words, \((C,\Comult[C])\) is monoidally equivalent to
\(A_o(F)\otimes \Cont(\T)\) (still if \(d=0\)).  It is already known
that \(A_o(F)\) is monoidally equivalent to \(\qSU\) for a unique
\(q \in [-1,1]\setminus\{0\}\)
(see~\cite{Bichon-de_Rijdt-Vaes:Ergodic}).  Hence~\((C,\Comult[C])\)
is monoidally equivalent to \(\qSU\otimes \Cont(\T)\) for the
same~\(q\).  This finishes the proof of
Theorem~\ref{the:monoidal_equivalence} in case \(d=0\).

\begin{remark}
  The equivalence of categories above does not preserve the functor
  to the underlying Hilbert spaces because the intertwiner
  \(z \tenscorep t \cong t \tenscorep z\) is not just the flip, but
  a non-trivial braiding operator.  Hence Woronowicz's
  Tannaka--Krein Theorem for compact quantum groups
  from~\cite{Woronowicz:Tannaka-Krein} does not apply and it does
  not follow that the quantum group~\((C,\Comult[C])\) would be
  isomorphic to \(\qSU\otimes \Cont(\T)\).
\end{remark}

Finally, we remove the assumption \(d=0\).  First, if~\(d\) is even,
then we use the \(\T\)\nb-equivariant Hopf \Star{}isomorphism of
Lemma~\ref{lem:shift_degrees} with the shift \(s=d/2\) to reduce to
the case \(d=0\).  Let \(\pi' = z^{-d/2} \pi\) and
\(F_{j i} = \omega'_{i j} = \zeta^{-d d_j/2} \omega_{i j}\) for
\(1 \le i,j \le n\).  Then \(A_o(V,\pi,\omega) \cong A_o(V,\pi',\omega')\)
as a braided quantum group, and hence the bosonisations are also isomorphic.  
Thus the description of the irreducible
representations and the monoidal equivalence to
\(A_o(F)\otimes \Cont(\T)\) or to \(\qSU\otimes \Cont(\T)\) for a
unique \(q\in [-1,1]\setminus\{0\}\) carry over to all
\(A_o(V,\pi,\omega)\) with even~\(d\).

If~\(d\) is odd, then we pass to a double cover of~\(\T\) as in
Section~\ref{sec:construct}.  More precisely, we replace the
representation~\(\pi\) of~\(\T\) by \(\pi''_z \defeq \pi_{z^2}\),
take the fourth root of~\(\zeta\), and the same vector~\(\omega\).  The
fundamental representation~\(t\) of \(A_o(V,\pi,\omega)\) corresponds
to the representation
\((z'')^{-d}\tenscorep t'' = (z'')^{-d} \tenscorep r_1\) of
\(A_o(V,\pi'',\omega)\), and the character~\(z\) on \(A_o(V,\pi,\omega)\)
corresponds to the character~\((z'')^2\).  Since any representation
of \(A_o(V,\pi,\omega)\) is contained in a tensor power of the two
fundamental representations \(z\) and~\(t\), we see that the
irreducible representations of \(A_o(V,\pi'',\omega)\) that come from
irreducible representations of \(A_o(V,\pi,\omega)\) are
\(z^a \tenscorep \varphi_*(r_b)\) for which \(b-a\) is even.  The
representation category of \(A_o(V,\pi,\omega)\) is the full monoidal
subcategory of \(\Rep(A_o(V,\pi'',\omega))\) whose objects are
isomorphic to direct sums of these particular irreducible
representations.  This allows to carry over the results proven for
\(d=0\) to the case of odd~\(d\) as well.  This gives Theorems
\ref{the:irrep_tensor} and~\ref{the:monoidal_equivalence} in
complete generality; \(\Rep(A_o(V,\pi,\omega))\) and \(\Rep(\qU)\)
correspond to the same subcategory of representations of
\(\qSU \times \Cont(\T)\) because our construction gives the same
subcategory for all \(A_o(V,\pi,\omega)\) with odd~\(d\), including
\(\qSU\).

\end{document}